\documentclass[reqno,11pt]{amsart}

\usepackage{mathbbold}
\usepackage{amscd}
\usepackage{amssymb}
\usepackage{amsmath}
\usepackage{amsfonts}
\usepackage{amssymb, latexsym, amsmath, pb-diagram}
\usepackage[all,cmtip]{xy}
\usepackage{enumerate}
\usepackage{epic,eepic}
\usepackage{galois}
\xyoption{all}

\def\w{\widetilde}

\newcommand{\zz}{\mathbb{Z}}

\newcommand{\RR}{\mathbb{R}}

\newtheorem{theorem}{Theorem}[section]
\newtheorem{prop}[theorem]{Proposition}
\newtheorem{lemma}[theorem]{Lemma}
\newtheorem{cor}[theorem]{Corollary}

\theoremstyle{definition}
\newtheorem{definition}[theorem]{Definition}

\newtheorem{exmp}[theorem]{Example}
\newtheorem{cons}[theorem]{Construction}
\newtheorem{rem}[theorem]{Remark}
\theoremstyle{remark}

\newtheorem*{conv*}{Convention}

\numberwithin{equation}{section}

\begin{document}
\title{Simplicial (co)homeology groups: New \emph{PL} homeomorphsm invariants of polyhedra}
\author[F.~Fan \& Q.~Zheng]{Feifei Fan and Qibing zheng}
\thanks{The authors are supported by NSFC grant No. 11261062, No. 11471167 and SRFDP No. 20120031110025}
\address{Feifei Fan, School of Mathematical Sciences and LPMC, Nankai University, Tianjin 300071, P.~R.~China}
\email{fanfeifei@mail.nankai.edu.cn}
\address{Qibing Zheng, School of Mathematical Sciences and LPMC, Nankai University, Tianjin 300071, P.~R.~China}
\email{zhengqb@nankai.edu.com}
\subjclass[2010]{Primary 55U10, 57Q05; Secondary 55T05, 51M20}

\begin{abstract}
In this paper, we define (reduced) homeology groups and (reduced) cohomeology groups on finite simpicial complexes and prove that these groups are $PL$ homeomorphsm invariants of polyhedra, while they are not homotopy invariants. So these groups can reflect some information that (co)homology groups can not tell. We also define homeotopy type of polyhedra which is finer than homotopy type but coarser than homeomorphism class, and prove that (co)homeology groups are actually homeotopy invariants. In the last section of this paper, we give a geometric description of some special (co)homeology groups.
\end{abstract}
\keywords{simplicial complexes, polyhedra, stellar subdivisions, double complexes, PL homeomorphismes}
\maketitle
\section{Introduction}

Simplicial complexes or triangulations (first introduced by Poincar\'e) provide an elegant, rigorous and convenient tool for studing topological invariants by
combinatorial methods \cite{A30, H69}. The algebraic topology itself evolved from studing triangulations of topological spaces. Simplicial theory have always played a significant role in $PL$ topology, discrete and combinatorial geometry. But so far, there are only algebraic homotopy invariants of polyhedra (geometrical realizations of simplicial complexes), such as homology and cohomology groups, and combinatorial invariants that are not homeomorphism invariants, such as the Tor-algebras of face rings.

In \S\ref{sec:2}, corresponding to each simplicial complex $K$, we define a cohomeology spectral sequence $\{E_r(K),\Delta_r\}$ (resp. homeology spectral sequence $\{E^r(K),D_r\}$), whose $E_2$-term, denoted by $\mathcal {H}^{p,\,q}(K)$ (resp. $\mathcal {H}_{p,\,q}(K)$), are called the cohomeology groups (resp. homeology groups) of $K$. We prove that these groups are $PL$ homeomorphism invariants of polyhedra (\textbf{Corollary \ref{cor:1}}), but not homotopy invariants. In fact, cohomology (resp. homology) groups is the convergence of the cohomeology (resp. homeology) spectral sequence (\textbf{Proposition \ref{prop:1}}).

In \S\ref{sec:3}, we define a category whose objects are polyhedra and morphisms are non-degenerate maps.
Then we prove that $\mathcal {H}_{*,*}$ (resp. $\mathcal {H}^{*,*}$) is a covariant (resp. contravariant) functor from polyhedra and non-degenerate maps to groups and group homomorphisms.
(\textbf{Theorem \ref{thm:2}}).
Homeotopy is defined between non-degenerate maps and homeotopy type of polyhedra is defined by homeotopy equivalence. We prove that (co)homeology groups are actually homeotopy invariants of polyhedra (\textbf{Theorem \ref{thx}}). Homeotopy types of polyhedra are finer than homotopy types but coarser than homeomorphism classes. For example, (co)homeology groups distinguish disks of different dimension and many other contractible polyhedra.

In \S\ref{sec:4}, we discuss the geometric meanings of cohomeology groups $\mathcal {H}^{p,\,q}(K)$ for some special $p$ and $q$. For example, like $\mathrm{rank}\,H^0(X)$ equals the number of the path-connected components of $X$, $\mathrm{rank}\,\mathcal {H}^{n,n}(K)$ equals the number of the $n$-dimensional completely connected component
(see Definition \ref{def:2}) of $K$.

\section{Preliminaries}\label{sec:1}
The following definition is a brief review of finite simplicial complex theory.
\begin{definition}
 A (finite, abstract) \emph{simplicial complex} $K$ with vertex set $\mathcal {S}$
 is a collection of subsets of the finite set $S$ satisfying the following two conditions.
\begin{enumerate}[(1)]
\item For any $v\in \mathcal {S}$, $\{v\}\in K$.

\item If $\sigma\in K$ and $\tau\subset\sigma$, then $\tau\in K$. Especially, the empty set $\varnothing\in K$.
\end{enumerate}

An subset $\sigma\in K$ is called a \emph{simplex} (or \emph{face}) of $K$. A maximal simplex is also called a \emph{facet}.
The \emph{dimension} of a simplex $\sigma\in K$ is its cardinality minus one: dim$\,\sigma=|\sigma|-1$. The dimension of $K$ is the maximal dimension of its
simplices, denoted dim$\,K$. $K$ is said to be \emph{pure} if all its facets have the same dimension.

The \emph{geometrical realization} (also called \emph{polyhedron}) of a simplicial complex $K$ is the topological space $|K|$ defined as follows. Suppose the vertex set $S=\{v_1,\dots,v_m\}$. Let $e_i\in\RR^m$ be such that the $i$-th coordinate of $e_i$ is $1$ and all other coordinates of $e_i$ are $0$. $|K|$ is the union of all the convex hull of $\{e_{i_0},\dots,e_{i_s}\}$ such that $\{v_{i_0},\dots,v_{i_s}\}$ is a simplex of $K$.
To avoid ambiguity of notations, we use $\langle\sigma \rangle$ to denote the geometrical realization of a simplex $\sigma\in K$.

A \emph{simplicial subcomplex} $L$ of $K$ is a subset $L\subset K$ such that $L$ is also a simplicial complex.

For simplicial complexes $K_1$ and $K_2$ with vertex sets $S_1$ and $S_2$ resp., a \emph{simplicial map} $f\colon K_1\to K_2$ is a map $f\colon S_1\to S_2$ such that for all $\sigma\in K$, $f(\sigma)\in L$ (On the geometric level, every simplicial map extends linearly to a map $\phi:|K|\to|L|$).
$K_1$ is said to be \emph{simplicial isomorphic} to $K_2$ if there are simplicial maps $f\colon K_1\to K_2$ and $g\colon K_2\to K_1$ such that $gf=1_{K_1}$ and $fg=1_{K_2}$.

For two polyhedra $P_1$ and $P_2$, a \emph{$PL$ map} $\phi:P_1\to P_2$ is a map that is simplicial between some subdivisions of $P_1$ and $P_2$.
A \emph{PL homeomorphism} is a \emph{PL} map for which there exists a \emph{PL} inverse.
\end{definition}

\begin{definition}
For a simplex $\sigma$ of $K$, the \emph{link} and the \emph{star} of $\sigma$ are the simplicial subcomplexes
\begin{align*}
\mathrm{link}_K\sigma&=\{\tau\in K:\sigma\cup\tau\in K,\,\sigma\cap\tau=\emptyset\};\\
\mathrm{star}_K\sigma&=\{\tau\in K:\sigma\cup\tau\in K\}.
\end{align*}
\end{definition}

\begin{definition}
Let $\sigma\in K$ be a nonempty simplex of a simplicial complex $K$.
The \emph{stellar subdivision} of $K$ at $\sigma$ is obtained by replacing the star of $\sigma$ by the cone
over its boundary:
\[\mathrm{S}_\sigma K=(K\setminus \mathrm{star}_K\sigma)\cup \big(\mathrm{cone}(\partial\sigma*\mathrm{link}_K\sigma)\big).\]
The symbol $*$ denotes the join of two complex. If $\mathrm{dim}\sigma=0$ then $\mathrm{S}_\sigma K=K$. Otherwise the complex $\mathrm{S}_\sigma K$ acquires an additional vertex (the vertex of the cone). $K$ is called a \emph{stellar weld} of $\mathrm{S}_\sigma K$.

Two simplicial complexes $K,L$ are called \emph{stellar equivalent}, if there is a sequence of simplicial complexes $K=K_0,K_1,\dots,K_n=L$ such that $K_{i+1}$ is either a stellar subdivision or a stellar weld of $K_i$ for $0\leqslant i< n$.
\end{definition}

The stellar subdivision is one of the standard tools in the theory of simplicial complexes and has an old and rich tradition. Later on we need the following fundamental theorem.
\begin{theorem}\label{th0}
Let $K,L$ be two simplicial complexes. Then $|K|$ $PL$ homeomorphic to $|L|$ if and only if $K$ and $L$ are stellar equivalent (see \cite[Theorem 4.5]{L99}).
\end{theorem}

The following definition is a brief review of simplicial homology and cohomology theory. We always regard the simplicial chain complex $(C_*(K),d)$ and its dual cochain complex $(C^*(K),\delta)$ as the same chain group $C_*(K)=C^*(K)$ with different differentials $d$ and $\delta$. This trick is essential for the definition of (co)homeology spectral sequence.

\begin{definition}
Let $K$ be a simplicial complex with vertex set $\mathcal {S}$.  The $n$-th \emph{chain group} $C_n(K)=C^n(K)$ of $K$ is the free abelian group generated by all ordered sequence $(v_0,v_1\dots,v_{n})$ of elements of $\mathcal {S}$ with the following relations.
\begin{enumerate}[(1)]
\item $(v_0,v_1,\cdots,v_{n})=0$ if $v_i=v_j$ for some $i\neq j$ or $(v_0,v_1,\dots,v_{n})\not\in K$,
\item $(\dots,v_{i},\dots,v_{j},\dots)=-(\dots,v_{j},\dots,v_{i},\dots)$ for all $i<j$.
\end{enumerate}

$(v_0,v_1,\dots,v_{n})\neq 0$ is called a \emph{chain simplex} of dimension $n$. Denote a generator of $C_{-1}(K)=\mathbb{Z}$ by $(\varnothing)$.

Set $C_*(K)=C^*(K)=\bigoplus_{k\geqslant 0}C_k(K)$.
The \emph{simplicial chain complex} $(C_*(K),d)$ and \emph{simplicial cochain complex} $(C^*(K),\delta)$ of $K$ are defined as follows. For any chain $k$-simplex $(v_0,v_1,\dots,v_{k})$,
\begin{align*}
d(v_0,v_1,\dots,v_{k})&=\sum_{i=1}^k(-1)^{i}(v_0,v_1,\dots,\hat v_{i},\dots,v_{k});\\
\delta(v_0,v_1,\dots,v_{k})&=\sum_{v\in S}(v,v_0,\dots,v_{k}).
\end{align*}
($\hat v_i$ means canceling the symbol from the term). $H_*(K)=H_*(C_*(K),d)$ and $H^*(K)=H^*(C^*(K),\delta)$ are resp. the \emph{homology} and \emph{cohomology} groups of $K$.

Set $\w C_*(K)=\w C^*(K)=\bigoplus_{k\geqslant -1}C_k(K)$. The \emph{reduced simplicial chain complex} $(\w C_*(K),d)$ and \emph{reduced simplicial cochain complex} $(\w C^*(K),\delta)$ of $K$ are defined similarly. The only differences are $d(v)=(\varnothing)$ and $\delta(\varnothing)=\sum_{v\in S}(v)$. $\w H_*(K)=H_*(\w C_*(K),d)$ and $\w H^*(K)=H^*(\w C^*(K),\delta)$ are resp. the \emph{reduced homology} and \emph{reduced cohomology} groups of $K$.
\end{definition}

\begin{conv*}
A chain simplex corresponds to a unique simplex. If there is no confusion, we use the same greek letter $\sigma,\tau,\cdots$ to denote both the chain simplex and the simplex it corresponds to. For two chain simplexes $\sigma=(v_1,\dots,v_m)$ and $\tau=(w_1,\dots,w_n)$, denote $\sigma*\tau=(v_1,\dots,v_m,w_1,\dots,w_n)$.
\end{conv*}

\section{(Co)Homeology spectral sequences}\label{sec:2}
\begin{definition}
A \emph{bicomplex} (or \emph{double complex}) is an ordered triple $(M, d', d'')$,
where $M=(M_{p,\,q})$ is a bigraded module, $d',\, d'' : M\to M$ are differentials of
bidegree $(-1, 0)$ and $(0,1)$ resp. (so that $d'd'=0$ and $d''d''=0$), and
\[d_{p,\,q+1}'d_{p,\, q}''+d_{p-1,\,q}''d_{p,\, q}'=0.\]

If $M$ is a bicomplex, then its \emph{total complex}, denoted by $(T^*(M),D)$,
is the complex with $n$th term
\[T^n(M)=\underset{q-p=n}{\bigoplus} M_{p,\,q}\]
and with differentials $D^n:T^n(M)\to T^{n+1}(M)$ given by
\[D^n=\sum_{q-p=n}d'_{p,\,q}+d''_{p,\,q}.\]

\end{definition}

As we know for any bicomplex $(M_{p,\,q})$, there are two filtrations of $T^*(M)$. The first is given by
\[\mathrm{^I}F^p\,(T^*(M))=\underset{i\leqslant p}{\bigoplus}M_{i,\,*}.\]
The second is given by
\[\mathrm{^{II}}F^p\,(T^*(M))=\underset{j\geqslant p}{\bigoplus}M_{*,\,j}.\]
Either of these filtrations yields a spectral sequence.

For a simplicial complex $K$, $C_*(K)\otimes C^*(K)$ can be seen as a bicomplex. That is
\[T^n(C_*(K)\otimes C^*(K))=\bigoplus_{q-p=n}C_p(K)\otimes C^q(K),\]
and differential $\Delta: C^{p,\,q}\to C^{p-1,\,q}\oplus C^{p,\,q+1}$ (where $C^{p,\,q}=C_p(K)\otimes C^q(K)$) is defined by \[\Delta(\sigma\otimes\tau)=(d\sigma)\otimes\tau+(-1)^{|\sigma|}\sigma\otimes(\delta\tau).\]

\begin{cons}
Define the bicomplex $(N^{*,*}(K),\Delta)$
to be the subcomplex of \[(C_*(K)\otimes C^*(K),\Delta)\] generated by all $\sigma\otimes\tau$ such that $\sigma\subset\tau$.
Similarly, the bicomplex $(\w N^{*,*}(K),\Delta)$
is defined to be  the subcomplex of $(\w C_*(K)\otimes\w C^*(K),\Delta)$ generated by all $\sigma\otimes\tau$ such that $\sigma\subset\tau$.

Dually, define \[(N_{*,*}(K),D)=((N^{*,*}(K))^*,\Delta^*)=({\rm Hom}_{\Bbb Z}(T^{*,*}(K),\Bbb Z),\Delta^*).\] It is easy to see that the differential $D$ is given by
\begin{eqnarray*}&&D\big((v_0,\cdots,v_m)\otimes(v_0,\cdots,v_m,v_{m+1},\cdots,v_{m+n})\big)\\
&=&\sum_{j=1}^n(-1)^j(v_0,\cdots,v_m)\otimes(v_0,\cdots,v_m,v_{m+1},\cdots,\hat v_{m+j},\cdots,v_{m+n})\\
&&+\sum_{j=1}^{n}(v_{m+j},v_0,\cdots,v_m)\otimes(v_0,\cdots,v_m,v_{m+1},\cdots,v_{m+n}).\end{eqnarray*}
\end{cons}
\begin{prop}\label{prop:1}
For any simplicial complex $K$, we have
\[H^*\big(T^*(N(K)),\Delta\big)\cong H^*(K);\quad H_*\big(T_*(N(K)),D\big)\cong H_*(K),\]
and
\[\quad H^*\big(T^*(\w N(K)),\Delta\big)\cong \mathbb{Z};\quad \quad H_*\big(T_*(\w N(K)),D\big)\cong \mathbb{Z}.\]
\end{prop}
\begin{proof}
We only prove the cases for $N^{*,*}(K)$ and $\w N^{*,*}(K)$. The others are dual.

From the filtration $\mathrm{^{II}}F^p=\bigoplus_{i\geqslant p}N^{*,i}(K)$ of $T^*(N(K))$, we get a spectral sequence $(\mathrm{^{II}}E_r,\partial_r)_{r\geqslant 1}$ with \[\mathrm{^{II}}E_{\,1}^{p,\,q}=H^{p-q}(\mathrm{^{II}}F^p/\mathrm{^{II}}F^{p+1}),
\ \partial_r\colon \mathrm{^{II}}E_{\,r}^{p,\,q}\to \mathrm{^{II}}E_{\,r}^{p+r,\,q+r-1}.\] By definition,
\[\mathrm{^{II}}F^p/\mathrm{^{II}}F^{p+1}\cong \bigoplus_{|\sigma|=p+1}C_*(2^{\sigma}).\]
So
\[\mathrm{^{II}}E_{\,1}^{p,\,q}=H^{p-q}(\mathrm{^{II}}F^p/\mathrm{^{II}}F^{p+1})\cong \bigoplus_{|\sigma|=p+1}H_q(2^{\sigma}).\]
Hence $\mathrm{^{II}}E_{\,1}^{p,0}\cong C^p(K)$ and $\mathrm{^{II}}E_{\,1}^{p,\,q}=0$ if $q\neq0$. On the other hand, an easy observation shows that
\[(\mathrm{^{II}}E_{\,1}^{*,0},\,\partial_1)\cong (C^*(K),\,\delta).\]
Thus $\mathrm{^{II}}E_{\,2}^{p,0}\cong H^p(K)$, $\mathrm{^{II}}E_{\,2}^{p,\,q}=0$ if $q\neq0$, and so the spectral sequence collapses at $E_2$-term. So
\[H^p\big(T^*(N(K)),\Delta\big)\cong \mathrm{^{II}}E_{\,2}^{p,0}\cong H^p(K).\]

Similarly, from the filtration $\mathrm{^{II}}\w F^p=\bigoplus_{i\geqslant p}\w N^{*,i}(K)$ of $T^*(\w N(K))$,
we get a spectral sequence $(\mathrm{^{II}}\w E_r,\,\partial_r)_{r\geqslant 1}$ with
\[\mathrm{^{II}}\w E_{\,1}^{p,\,q}=H^{p-q}(\mathrm{^{II}}\w F^p/\mathrm{^{II}}\w F^{p+1})\cong \bigoplus_{|\sigma|=p+1}\w H_q(2^{\sigma}).\]
So $\w E_1^{-1,-1}\cong \mathbb{Z}$ and $\w E_1^{p,\,q}(K)=0$ otherwise. Thus the spectral sequence collapses at $E_1$-term. Then we have
\[H^{*}\big(T^*(\w N(K)),\Delta\big)\cong \mathrm{^{II}}\w E_{\,1}^{-1,-1}\cong \mathbb{Z}.\]
\end{proof}
\begin{definition}\label{def:0}
The \emph{cohomeology spectral sequence} of $K$
\[(E_r(K),\,\Delta_r)_{r\geqslant 1},\quad \Delta_r\colon E_{\,r}^{p,\,q}(K)\to E_{\,r}^{p-r,\,q-r+1}(K)\]
is yielded by the filtration
$F^{p}=\bigoplus_{i\leqslant p}\, N^{i,\,*}(K)$ of $T^*(N(K))$, where
\[E_{\,1}^{p,\,q}(K)=H^{q-p}(F^{p}/F^{p-1}).\]

Denote the $E_2$-term $E_2^{*,*}(K)$ by $\mathcal {H}^{*,*}(K)$, called the \emph{cohomeology groups} of $K$. For an abelian group $G$, define $ N^{*,*}(K;G)= N^{*,*}(K)\otimes G$. Then we get the cohomeology groups $\mathcal {H}^{*,*}(K;G)$ of $K$ with coefficients in $G$.

The \emph{reduced cohomeology spectral sequence} of $K$
\[(\w E_r(K),\,\Delta_r)_{r\geqslant 1},\quad \Delta_r\colon \w E_{\,r}^{p,\,q}(K)\to \w E_{\,r}^{p-r,\,q-r+1}(K)\]
is yielded by the filtration $\w F^{p}=\bigoplus_{i\leqslant p}\,\w N^{i,*}(K)$.
Denote $\w E_2^{*,*}(K)$ by $\w{\mathcal {H}}^{*,*}(K)$, called the \emph{reduced cohomeology groups} of $K$. Similarly we have the definition of
$\w{\mathcal {H}}^{*,*}(K;G)$.

The \emph{homeology spectral sequence} of $K$
\[(E^r(K),\,D_r)_{r\geqslant 1},\quad D_r\colon E^{\,r}_{p,\,q}(K)\to E^{\,r}_{p+r,\,q+r-1}(K)\]
is yielded by the filtration $F_{p}=\bigoplus_{i\geqslant p}\, N_{i,\,*}$, where
\[E^1_{p,\,q}(K)=H_{q-p}(F_{p}/F_{p+1}).\]
Denote $E^2_{*,*}(K)$ by $\mathcal {H}_{*,*}(K)$, and called the \emph{homeology groups} of $K$. Similarly we have the \emph{reduced homeology groups}
$\w{\mathcal {H}}_{*,*}(K)$ (and $\mathcal {H}_{*,*}(K;G)$, $\w{\mathcal {H}}_{*,*}(K;G)$).
\end{definition}

\begin{lemma}
Let $K$ be a simplicial complex. Suppose $\sigma\in K$, $v\in\sigma$, $\sigma'=\sigma\setminus\{v\}$. Then there is a cochain homomorphism
\[\phi_v\colon(\w C^*({\rm link}_K\sigma),\delta)\to(\w C^{*+1}({\rm link}_K\sigma'),\delta)\]
generated by $\phi_v(\tau)=\tau*(v)$ for all $\tau\in{\rm link}_K\sigma$.
\end{lemma}
\begin{proof}
We need only verify that $\phi_v$ is commutative with the differential $\delta$. Let $\mathcal {S}$ and $\mathcal {S}'$ are the vertex sets of
$\mathrm{link}_K\sigma$ and $\mathrm{link}_K\sigma'$ respectively. $\mathcal {S}\subset \mathcal {S}'$ is clear. By definition of $\phi_v$ and $\delta$
\[\delta\phi_v(\tau)=\sum_{u\in S'}(u)*\tau*(v),\]
and
\[\phi_v\delta(\tau)=\sum_{u\in S}(u)*\tau*(v).\]
If $u\in \mathcal {S}'$ but $u\not\in \mathcal {S}$, then $\{u\}\cup\sigma\not\in K$, and then $\{u\}\cup\{v\}\not\in\mathrm{link}_K\sigma'$,
so $(u)*\tau*(v)=0$ in $C^*(\mathrm{link}_K\sigma')$. The lemma follows immediately.
\end{proof}
Thus $\phi_v$ induces a homomorphism $\w H^*({\rm link}_K\sigma)\to \w H^{*+1}({\rm link}_K\sigma')$, which we also denote by $\phi_v$.
Denote by  $d_v\colon(C_*({\rm link}_K\sigma'),d)\to(C_{*-1}({\rm link}_K\sigma),d)$ the dual homomorphism of $\phi_v$.

\begin{lemma}\label{lem:1}
Let $K$ be a simplicial complex.  The cohomeology and homeology spectral sequences of $K$ satisfy that
\begin{align*}
E_1^{p,\,q}(K)&\cong\bigoplus_{|\sigma|=p+1}\,\w H^{q-p-1}({\rm link}_K\sigma)\\
E^1_{p,\,q}(K)&\cong\bigoplus_{|\sigma|=p+1}\,\w H_{q-p-1}({\rm link}_K\sigma)
\end{align*}
The differentials $\Delta_1$ and $D_1$ are defined as follows.

Given a cohomology class
$[c]\in\w H^{q-|\sigma|}({\rm link}_K\sigma)$, then $\Delta_1([c])=(-1)^q\cdot\sum_{v\in\sigma}[\phi_v(c)]$.
Given a homology class $[c]\in\w H_{q-|\sigma|}({\rm link}_K\sigma)$, then $D_1([c])=(-1)^q\cdot\sum_v[d_v(c)]$,
where the sum is taken over all vertices $v$ such that $\sigma\cup \{v\}\in K$.
\end{lemma}
\begin{proof}
We only prove the cohomeology case. The homeology case is totally dual.

Note that $E_1^{p,\,q}(K)=H^{q-p}(F^{p}/F^{p-1})$. Apparently there are isomorphisms of chain complexes
\[(F^{p}/F^{p-1},\Delta)\cong (N^{p,*},1\otimes\delta),\]
 and
 \[(N^{p,*},1\otimes\delta)\cong  \bigoplus_{|\sigma|=p+1}(N^{\sigma,*},1\otimes\delta),\]
 where $N^{\sigma,*}$ is a subgroup of $N^{p,*}$ generated by $\sigma\otimes\tau$ with $\sigma\subset \tau$.
We can construct a isomorphism of chain complexes
\[(N^{\sigma,*},1\otimes\delta)\cong (C^{*-|\sigma|}(\mathrm{link}_K\sigma),\delta)\]
generated by
$\sigma\otimes(\sigma*\sigma')\mapsto\sigma'$. So the isomorphisms in the theorem hold.

On the other hand, notice that $\Delta_1:E_1^{p,\,q}(K)\to E_1^{p-1,\,q}(K)$
is just
\[d\otimes1: H^{q-p}(F^{p}/F^{p-1})\to H^{q-p+1}(F^{p-1}/F^{p-2}).\]
It is easily verified that
\[d\otimes1:H^{q-p}(N^{p,*},1\otimes\delta))\to H^{q-p+1}(N^{p-1,*},1\otimes\delta)\]
is equivalent to
\[\bigoplus_{|\sigma|=p+1}(-1)^q\cdot\sum_{v\in\sigma}\phi_v:\bigoplus_{|\sigma|=p+1}\,\w H^{q-|\sigma|}({\rm link}_K\sigma)\to \bigoplus_{|\tau|=p}\,\w H^{q-|\tau|}({\rm link}_K\tau)\]
Then the statement of lemma \ref{lem:1} follows.
\end{proof}
\begin{theorem}\label{thm:3}
Let $K$ be a simplicial complex, $K'$ be the stellar subdivision of $K$ at a simplex of dimension greater than zero.  Then for any abelian group $G$
and all $p,q\in \mathbb{Z}$,
\begin{eqnarray*}
&\mathcal {H}^{p,\,q}(K';G)\cong\mathcal {H}^{p,\,q}(K;G),\quad \w {\mathcal {H}}^{p,\,q}(K';G)\cong\w {\mathcal {H}}^{p,\,q}(K;G),\\
&\mathcal {H}_{p,\,q}(K';G)\cong\mathcal {H}_{p,\,q}(K;G),\quad \w {\mathcal {H}}_{p,\,q}(K';G)\cong\w {\mathcal {H}}_{p,\,q}(K;G).
\end{eqnarray*}
\end{theorem}

\begin{proof}
We only prove the unreduced cohomeology case. Other cases are similar. For convention we ignore the coefficient group $G$.

Define $(\mathcal {N}^{*,*},\Delta)$ to be the double subcomplex of $(C_*(K'),d)\otimes(C^*(K),\delta)$ generated by all $\sigma'\otimes\sigma$ such that $\sigma'\in K'$, $\sigma\in K$ and $\langle\sigma' \rangle\subset\langle\sigma \rangle $.
From the definition of stellar subdivision, for each $\sigma'\in K'$, there exist one and only one simplex $f(\sigma')\in K$ such that $\overset{\comp }{\langle\sigma' \rangle}\subset\overset{\comp }{\langle f(\sigma') \rangle}$, where
$\overset{\comp }{\langle\cdot \rangle}$ denote the interior of the geometric realization of a simplex.

Define a group homomorphism
\[\varphi: N^{*,*}(K')\to \mathcal {N}^{*,*}\]
generated by
\[
\varphi(\sigma'\otimes\tau')=
\begin{cases}
\pm\sigma'\otimes f(\tau')&\ \text{ if }  \mathrm{dim}\,\tau'=\mathrm{dim}\,f(\tau'),\\
0&\ \text{ otherwise.}
\end{cases}
\]
The sign in the above formula depends on whether $\langle\tau' \rangle$ has the same orientation as $\langle f(\tau') \rangle$. It is easy to check that $\varphi$ is a
chain homomorphism.
The filtration $\mathcal {F}^{p}=\oplus_{i\leqslant p}\,\mathcal {N}^{i,*}$ of $T^*(\mathcal {N})$ also yields a spectral sequence
$(\mathcal {E}_r,\,\Delta_r)_{r\geqslant 1}$, and so $\varphi$ induces a spectral sequence homomorphism
$\varphi_r\colon E_{\,r}^{p,\,q}(K')\to \mathcal {E}_{\,r}^{p,\,q}$, $r\geqslant 1$. Consider the following commutative diagram.
\[
\begin{CD}
E_{\,1}^{p,\,q}(K')@>\varphi_1>>\mathcal {E}_{\,1}^{p,\,q}\\
@V\cong V\xi V @V\cong V\zeta V\\
\underset{\substack{\sigma'\in K'\\|\sigma'|=p+1}}{\bigoplus}\w H^{q-p-1}({\rm link}_{K'}\sigma')@>
\phi_1>>\underset{\substack{\sigma'\in K'\\|\sigma'|=p+1}}{\bigoplus}H^q\big(C^*_{\sigma'}(K)\big)
\end{CD}\]
where $\phi_1=\zeta\varphi_1\xi^{-1}$, $C^*_{\sigma'}(K)$ is the subcomplex of $C^*(K)$ generated by
\[\Gamma_{\sigma'}=\{\sigma\in K: \langle\sigma' \rangle\subset\langle\sigma \rangle\} \]
For each $\sigma\in \Gamma_{\sigma'}$, denote $e_{\sigma'}(\sigma)=|{\rm link}_{K'}\sigma'|\cap \langle\sigma \rangle$.
Then it is easily verified that
\[
e_{\sigma'}(\sigma)\cong
\begin{cases}
S^{|\sigma|-|\sigma'|-1}&\ \text{ if }\ \overset{\comp }{\langle\sigma' \rangle}\subset\overset{\comp }{\langle\sigma\rangle},\\ D^{|\sigma|-|\sigma'|-1}&\ \text{ otherwise. }
\end{cases}
\]
Thus ${\rm link}_{K'}\sigma'$ can be viewed as a CW complex $X_{\sigma'}$ with $e_{\sigma'}(\sigma)$ as cells
(if $\overset{\comp }{\langle\sigma' \rangle}\subset\overset{\comp }{\langle\sigma\rangle}$, take a vertex $v\in e_{\sigma'}(\sigma)$ as an additioanl $0$-cell).
Apparently, there is an isomorphism
\[\eta: H^q\big(C^*_{\sigma'}(K)\big)\cong \w H^{q-|\sigma'|}(X_{\sigma'})\] induced by the cochain homomorphism generated by $\sigma\mapsto e_{\sigma'}(\sigma)$. It is easy to see that $\eta\circ h$ is the canonical isomorphism $\w H^*({\rm link}_{K'}\sigma')\cong \w H^*(X_{\sigma'})$. So $\varphi_1$ is an isomorphism and inductively,
$\varphi_r: E_{\,r}^{p,\,q}(K')\to\mathcal {E}_{\,r}^{p,\,q}$ are isomorphisms for $r\geqslant 1$.

Define a monomorphism of double complexes
\[\psi\colon (N^{*,*}(K),\Delta)\to (\mathcal {N}^{*,*},\Delta)\] generated by
$\psi(\sigma\otimes\tau)=\sum_{\sigma'}\sigma'\otimes\tau$, the sum taken over all oriented simplex $\sigma'\in K'$ such that $f(\sigma')=\sigma$, $|\sigma'|=|\sigma|$
and $\langle\sigma' \rangle$ has the same orientation as $\langle \sigma \rangle$. $\psi$ induces a spectral sequence homomorphism
$\psi_r\colon E_{\,r}^{p,\,q}(K)\to \mathcal {E}_{\,r}^{p,\,q}$, $r\geqslant 1$.

Denote the quotient double complex of $\psi$ by
$(\bar N^{*,*},\Delta)$.
The filtration $\bar F^{p}=\bigoplus_{i\leqslant p}\,\bar N^{i,\,*}$ of $T^*(\bar N)$ yields a spectral sequence
$(\bar E_r,\,\Delta_r)_{r\geqslant 1}$. There are long exact sequences for any fixed $p$,
\[\cdots\to E_{\,1}^{p,\,q}(K)\xrightarrow{\psi_1} \mathcal {E}_{\,1}^{p,\,q}\to \bar E_{\, 1}^{p,\,q}\to E_{\,1}^{p,\,q+1}(K)\to \cdots\]
associated to the short exact sequences of chain complexes
\[ 0\to F^p/F^{p-1}\xrightarrow{\psi}\mathcal {F}^p/\mathcal {F}^{p-1}\to \bar F^p/\bar F^{p-1}\to 0.\]
For each $\sigma\in K$, if $f(\sigma')=\sigma$, then there is apparently a cochain isomorphism
\[(N^{\sigma,*}(K),\,\delta\otimes1) \cong (C^*_{\sigma'}(K),\,\delta),\]
which induces a cohomology isomorphism
\[\lambda_{\sigma'}^\sigma: H^*\big(N^{\sigma,*}(K),\,\delta\otimes1\big)\to H^*\big(C^*_{\sigma'}(K),\,\delta\big).\]
By the definition of $\psi$, for a cohomology class $[x]\in H^*(N^{\sigma,*}(K),\,\delta\otimes1)$,
\[\psi_1([x])=\sum_{f(\sigma')=\sigma,\,|\sigma'|=|\sigma|}\lambda_{\sigma'}^\sigma([x]).\]
So $\psi_1$ is injective, and the long exact sequences above break up into short exact sequences
\[0\to E_{\,1}^{p,\,q}(K)\xrightarrow{\psi_1} \mathcal {E}_{\,1}^{p,\,q}\to \bar E_{\, 1}^{p,\,q}\to 0.\]
Now Fix $q$, then we get a short exact sequence of chain complexes
\[0\to (E_{\,1}^{*,\,q}(K),\Delta_1)\xrightarrow{\psi_1} (\mathcal {E}_{\,1}^{*,\,q},\Delta_1)\to (\bar E_{\, 1}^{*,\,q},\Delta_1)\to 0.\]
This short exact sequence of chain complexes yields a long exact sequence
\[\cdots\to \mathcal {H}^{p,\,q}(K)\xrightarrow{\psi_2} \mathcal {E}_{\,2}^{p,\,q}\to \bar E_{\, 2}^{p,\,q}\to \mathcal {H}^{p-1,\,q}(K)\to \cdots\]
If we prove that $\bar E_{2}^{p,\,q}=0$ for all $p,q\in \mathbb{Z}$, then $\psi_2$ is an isomorphism, and therefore,
\[\varphi_2^{-1}\psi_2\colon \mathcal {H}^{*,*}(K)\to \mathcal {H}^{*,*}(K')\] is an isomorphism.

Now we prove $\bar E_{2}^{*,*}=0$. Let $K^{(n)}$ be the $n$-skeleton of $K$. Fix $q$, define a filtration
$(F_q^p)_{p\in \mathbb{Z}}$ of $(E_{1}^{*,\,q}(K),\,\Delta_1)$ by
\[F_q^p=\bigoplus_{i\leqslant p}E_{1}^{i,\,q}(K).\]
Meanwhile, define a filtration $(\mathcal {F}_q^p)_{p\in \mathbb{Z}}$ of $\mathcal {E}_{1}^{*,\,q}\cong \bigoplus_{\sigma'\in K'}H^q\big(C^*_{\sigma'}(K)\big)$ by
\[\mathcal {F}_q^p=\bigoplus_{\sigma'\in K',\langle\sigma' \rangle\subset |K^{(p)}|}H^q\big(C^*_{\sigma'}(K)\big).\]
Let $(\bar F_q^p)_{p\in \mathbb{Z}}$ be the quotient filtration of $\varphi_1$ on $\bar E_{\, 1}^{*,\,q}$, i.e.,
\[\bar F_q^p=\mathcal {F}_q^p/\psi_1(F_q^p).\]
Clearly, the spectral sequence $\{_q\bar E^{r}\}_{r\geqslant 1}$ (where $_q\bar E^{1}_{s,t}=H_{s+t}(\bar F_q^p/\bar F_q^{p-1})$) yielded by $(\bar F_q^p)_{p\in \mathbb{Z}}$ converges to
\[H_*(\bar E_{1}^{*,\,q},\Delta_1)\cong \bar E_{2}^{*,q}.\]

By the definition of $F_q^p$,
\[
(F_q^p/F_q^{p-1})_n=
\begin{cases}
E_{1}^{p,\,q}(K)\ &\text{ if } n=p,\\
0\ &\text{ otherwise.}
\end{cases}
\]
So \[H_*(F_q^p/F_q^{p-1})=H_p(F_q^p/F_q^{p-1})\cong E_{1}^{p,\,q}(K)\cong \underset{\substack{\sigma\in K,\\|\sigma|=p+1}}{\bigoplus}\w H^{q-p-1}(\mathrm{link}_K\sigma),\]
If $\overset{\comp }{\langle\sigma' \rangle}\subset\overset{\comp }{\langle\sigma\rangle}$ for $\sigma'\in K'$, $\sigma\in K$, then it is easily verified that
\[H^q\big(C^*_{\sigma'}(K)\big)\cong \w H^{q-|\sigma|}(\mathrm{link}_K\sigma).\]
Therefore we have
\[(\mathcal {F}_q^p/\mathcal {F}_q^{p-1})_{n}=\underset{\substack{\sigma\in K,\\|\sigma|=p+1}}{\bigoplus}C_n(K'_\sigma,\partial K'_\sigma)\otimes \w H^{q-p-1}(\mathrm{link}_K\sigma),\]
where $K'_\sigma$ is a subcomplex of $K'$ such that $|K'_\sigma|=\langle\sigma \rangle\cong D^{|\sigma|-1}$.
An easy verification shows that
\[H_*(\mathcal {F}_q^p/\mathcal {F}_q^{p-1})\cong \underset{\substack{\sigma\in K,\\|\sigma|=p+1}}{\bigoplus}H_*\big(C_*(K'_\sigma,\partial K'_\sigma)\big)\otimes\w H^{q-p-1}(\mathrm{link}_K\sigma).\]
Since $H_*\big(C_*(K'_\sigma,\partial K'_\sigma)\big)\cong H_{|\sigma|-1}(D^{|\sigma|-1},S^{|\sigma|-2})\cong G$, then from the the definition of $\psi_1$, we have that
\[(\psi_1)_*:H_*(F_q^p/F_q^{p-1})\xrightarrow{\cong}H_*(\mathcal {F}_q^p/\mathcal {F}_q^{p-1}).\]
Consider the long exact sequence
\[\cdots\to H_n(F_q^p/F_q^{p-1})\xrightarrow{(\psi_1)_*}H_n(\mathcal {F}_q^p/\mathcal {F}_q^{p-1})\to {_q\bar E}_{p,\,n-p}^1\to H_{n-1}(F_q^p/F_q^{p-1})\to\cdots.\]
We have that $_q\bar E^1=0$, and so $\bar E_{2}^{*,q}=0$. The proof is completed.
\end{proof}
Combine the results of Theorem \ref{th0} and Theorem \ref{thm:3}, we get the following corollary immediately.
\begin{cor}\label{cor:1}
$\mathcal {H}^{*,*}$, $\w {\mathcal {H}}^{*,*}$, $\mathcal {H}_{*,*}$, $\w {\mathcal {H}}_{*,*}$
 are all $PL$ homeomorphism invariants of polyhedra.
\end{cor}

\begin{theorem}\label{th5}
Let $K$ be an $n$-dimensional simplicial complex satisfies that for any simplex $\sigma\in K$ and $\sigma\neq\varnothing$, $\w H^i({\rm link}_K\sigma)=0$ whenever $i\neq n-|\sigma|$. Then
\[\mathcal {H}^{p,\,q}(K;G),\, \mathcal {H}_{p,\,n}(K;G)=0\ \text{ for } q<n,\] and $\mathcal {H}^{p,\,n}(K;G)$ (resp. $\mathcal {H}_{p,\,n}(K;G)$) is
isomorphic to $H^{n-p}(K;G)$ (resp. $H_{n-p}(K;G)$).
\end{theorem}

\begin{proof}
We only prove the cohomeology case.

By Lemma \ref{lem:1}, $E^{p,\,q}_{\,1}(K;G)=0$ if $q\neq n$. Then the differential $\Delta_r$ in the cohomeology spectral sequence is zero for $r\geqslant 2$. So the cohomeology spectral sequence collapses at $E_2$-term. From Proposition \ref{prop:1}, we have that $\{E_{r}(K;G)\}$ converges to $H^*(K;G)$. Then the statement of Theorem \ref{th5} follows.
\end{proof}

\begin{rem}
If $K$ is a triangulation of a $n$-dimensional manifold $M^n$ with boundary $\partial M^n$, then $K$ is apparently satisfies the condition in Theorem \ref{th5}. Eventually, we can use Theorem \ref{th5} to prove Lefschetz duality theorem as follows.

Consider the following commutative diagram of complexes:
\[\begin{CD}
E_{\,1}^{n,n}(K)@>\Delta_1>> E_{\,1}^{n-1,n}(K)@>\Delta_1>>\cdots @>\Delta_1>> E_{\,1}^{0,n}(K)@>>>0\\
@Vf_nV\cong V @Vf_{n-1}V\cong V  && @Vf_{0}V\cong V \\
C_{n}(K,\partial K)@>d>>C_{n-1}(K,\partial K)@>d>>\cdots @>d>>C_{0}(K,\partial K)@>>>0
\end{CD}\]
where $f_i$ are isomorphisms defined as follows.
If $\sigma\in K\setminus\partial K$, then $|{\rm link}_K\sigma|\cong S^{n-|\sigma|}$, and then in this case
\[\w H^{*}({\rm link}_K\sigma)\cong\w H^{n-|\sigma|}(S^{n-|\sigma|})=\mathbb{Z}.\] Choose a cohomology class $[\tau]$ as the generator of
$H^{n-|\sigma|}({\rm link}_K\sigma)$ so that $\sigma*\tau$ corresponds to the orientation of $M^{n}$. Then define $f_{|\sigma|-1}([\tau])=\sigma.$
On the other hand, if $\sigma\in K\setminus\partial K$, then $|\mathrm{link}_K\sigma|\cong D^{n-|\sigma|}$, so $\w H^{*}(\mathrm{link}_K\sigma)=0$.
From Lemma \ref{lem:1}, it is easy to see that $f_i$ is an isomorphism for each $i\geqslant 0$. Therefore
$\mathcal {H}^{i,n}(K)\cong H_i(M^n,\partial M^n)$, and by applying Theorem \ref{th5}, we have
\[H^{n-i}(M^n)\cong H_i(M^n,\partial M^n),\]
which is exactly the form of Lefschetz duality theorem.
\end{rem}
\begin{theorem}\label{thm:5}
Let $K$ and $L$ be two simplicial complexes. Then
\begin{enumerate}[(a)]
\item $\mathcal {H}^{*,*}(K\sqcup L)\cong \mathcal {H}^{*,*}(K)\oplus\mathcal {H}^{*,*}( L)$, where $\sqcup$ means the disjoint union of sets.
\item If one of $\w E_{\,1}^{*,*}(K)$, $\w E_{\,1}^{*,*}(L)$ is torsion free and one of $\w{\mathcal {H}}^{*,*}(K)$, $\w{\mathcal {H}}^{*,*}( L)$ is torsion free, then
    \[\Sigma^{-1}\w{\mathcal {H}}^{*,*}(K*L)\cong \w{\mathcal {H}}^{*,*}(K)\otimes\w{\mathcal {H}}^{*,*}( L),\]
where $\Sigma^{-1}$ means lowering the bidegree by $(-1,-1)$.
\end{enumerate}

\end{theorem}
\begin{proof}
(a) is an immediate consequence of the isomorphism of double complexes
\[(N^{*,*}(K\sqcup L),\Delta)\cong (N^{*,*}(K),\Delta)\oplus (N^{*,*}(L),\Delta).\]

(b) The map \[(\sigma\otimes\tau)\otimes(\sigma'\otimes\tau')\mapsto (-1)^{|\tau|\cdot|\sigma'|}(\sigma*\sigma')\otimes(\tau*\tau')\] gives the following isomorphism of double complexes
\[(\w N^{*,*}(K),\Delta)\otimes(\w N^{*,*}(L),\Delta)\cong \Sigma^{-1} (\w N^{*,*}(K*L),\Delta).\]
Then the statement of (b) follows by K\"{u}nneth theorem.
\end{proof}

The homoelogy version of Theorem \ref{thm:5} is obtained in the same way.

At the end of this section, we calculate the cohomeology groups of some examples.

\begin{exmp}\label{exmp:1}
Both the $n$-disk $D^n$ and the $n$-sphere $S^n$ satisfy the condition in Theorem \ref{th5}, so
\[\mathcal {H}^{*,*}(D^n)=\mathcal {H}^{n,n}(D^n)=\mathbb{Z},\]
and \[\mathcal {H}^{0,n}(S^n)=\mathcal {H}^{n,n}(S^n)=\mathbb{Z},\ \mathcal {H}^{p,\,q}(S^n)=0 \text{ otherwise.}\]

As an application, cohomeology groups distinguishes disks of different dimensions, i.e., $D^n\not\cong D^m$ if $m\neq n$.
\end{exmp}

\begin{exmp}
Let $C_nK$ be the join of a simplicial complex $K$ with the $0$-dimensional simplicial complex consisting of $n$ isolated vertices (e.g., $C_1K=\mathrm{cone}K$, $C_2K=SK$). It is easy to verify that for any $0$-dimensional complex $L$,
\[\w {\mathcal {H}}^{*,*}(L)=\w {\mathcal {H}}^{0,0}(L)=\mathbb{Z}.\] Then by Theorem \ref{thm:5},
$\w {\mathcal {H}}^{p+1,\,q+1}(C_nK)=\w {\mathcal {H}}^{p,\,q}(K)$ for all $p,\,q\in \mathbb{Z}$ and $n\geqslant1$.
 Notice that $C_3K$ and $SK\vee SK$ have the same homotopy type. But in general,
\[\w {\mathcal {H}}^{*,*}(SK\vee SK)\neq\w {\mathcal {H}}^{*,*}(C_3K).\] For example, if $K=\partial\Delta^n$ with $n\geqslant2$, then $\w {\mathcal {H}}^{n,n}(SK\vee SK)=\zz\oplus\zz$ (ref. theorem \ref{th6}), whereas \[\w {\mathcal {H}}^{n,n}(C_3K)=\w {\mathcal {H}}^{n-1,n-1}(\partial\Delta^n)={\mathcal {H}}^{n-1,n-1}(\partial\Delta^n)=\zz.\] This shows that reduced cohomeology groups are not homotopy invariants.
\end{exmp}

\section{Block complex and its (co)homeology groups}
\begin{definition}
For a simplicial complex $K$, a \emph{block complex} $\mathcal {B}=\{b_i^n\}$ on $K$ is a collection $\mathcal {B}=\{b^n_i\}_{n\geqslant-1}$ of simplicial subcomplexes $b^n_i$ of $K$ that satisfies the following conditions.
\begin{enumerate}
\item Every $b_i^n$ is a triangulation of disk $D^n$. Especially, $b^{-1}=\varnothing$.
\item $\overset{\comp}{|b_i^m|}\bigcap\overset{\comp}{|b_j^n|}=\varnothing$ for any $b_i^m,\,b_j^n\in \mathcal {B}$.
\item $|K|=\bigcup_{\,b_i^n\in \mathcal {B}}\overset{\comp}{|b_i^n|}$.
\end{enumerate}
$b_i^n$ is called a \emph{$n$-block} of $\mathcal {B}$. $b_i^m$ is called a \emph{face} of $b_j^n$ if $b_i^m\subset b_j^n$.
\end{definition}

\begin{exmp}
$\mathcal {B}^0=\{2^{\sigma}\}_{\sigma\in K}$ is clearly a block complex on $K$ which is called the \emph{trivial block complex} on $K$.
\end{exmp}
\begin{exmp}
Let $K$ be a simplicial complex, $K'$ be a stellar subdivision of $K$. Then $\mathcal {B}=\{K'_\sigma\subset K':\sigma\in K,\, |K'_\sigma|=\langle\sigma \rangle\}$ is a block complex on $K'$.
\end{exmp}

\begin{definition} Let $\mathcal {B}$ be a block complex on a simplicial complex $K$. Let $C_n(\mathcal {B})=C^n(\mathcal {B})$ be the free abelian group generated by all $n$-blocks of $\mathcal {B}$. Set
\[C_*(\mathcal {B})=C^*(\mathcal {B})=\bigoplus_{k\geqslant 0}C_k(\mathcal {B}).\]
The \emph{block chain complex} $C_*(\mathcal {B},d)$ and \emph{block cochain complex} $C^*(\mathcal {B},\delta)$ of $\mathcal {B}$ are defined as follows.

Take an orietation on every block of $\mathcal {B}$. A chain $n$-simplex $\sigma$ of a given $n$-block $b_i^n$ is said to be positive if
$\langle\sigma \rangle$ has the same orientation as $|b_i^n|$, otherwise $\sigma$ is said to be negative.  For two blocks $b^n_i$ and $b^{n-1}_j$, the connecting coefficient $[b^n_i,b^{n-1}_j]$ is defined as follows. $[b^n_i,b^{n-1}_j]=0$ if $b^{n-1}_j$ is not a face of $b^n_i$; $[b^n_i,b^{n-1}_j]=1$ if there is a positive chain simplex $(v_0,v_1\dots,v_n)$ of $b^n_i$ such that $(v_1,\dots,v_{n})$ is a positive chain simplex of $b^{n-1}_j$; $[b^n_i,b^{n-1}_j]={-}1$ if there is a positive chain simplex $(v_0,v_1,\cdots,v_n)$ of $b^n_i$ such that $(v_1,\cdots,v_{n})$ is a negative chain simplex of $b^{n-1}_j$. It is obvious that the connecting coefficient is independent of the choice of the chain simplex. Then define
\begin{eqnarray*}
&&d(b^n_i)=\sum_{b^{n-1}_j\in \mathcal {B}}[b^n_i,b^{n-1}_j]b^{n-1}_j,\\
&&\delta(b^{n-1}_j)=\sum_{b^n_i\in \mathcal {B}}[b^n_i,b^{n-1}_j]b^n_i.
\end{eqnarray*}

Set\[\w C_*(\mathcal {B})=\w C^*(\mathcal {B})=\bigoplus_{k\geqslant -1} C_k(\mathcal {B}).\] The \emph{reduced block chain complex} $\w C_*(\mathcal {B},d)$ and \emph{reduced block cochain complex} $\w C^*(\mathcal {B},\delta)$ of
$\mathcal {B}$ are defined similarly. The only differences are $d(b^0_i)=(\varnothing)$, and  $\delta(\varnothing)=\sum_{b^0_i\in\mathcal {B}}b^0_i$.
\end{definition}

\begin{definition}
For a block complex $\mathcal {B}$ on a given simplicial complex $K$, the double complex $(N^{*,*}(\mathcal {B}),\Delta)$
is the double subcomplex of $(C_*(\mathcal {B}),d)\otimes(C^*(\mathcal {B}),\delta)$ generated by all $b^m_i\otimes b^n_j$ such that $b^m_i\subset b^n_j$.

The \emph{cohomeology spectral sequence} $(E_r(\mathcal {B}),\,\Delta_r)_{r\geqslant 1}$ and \emph{cohomeology groups} $\mathcal {H}^{*,*}(\mathcal {B})$ of $\mathcal {B}$
are defined in the same way as Definition \ref{def:0}.

In similar fashion we get the definition of $\w {\mathcal {H}}^{*,*}(\mathcal {B})$, $\mathcal {H}_{*,*}(\mathcal {B})$ and $\w {\mathcal {H}}_{*,*}(\mathcal {B})$.
\end{definition}

Using $\mathcal {B}$ instead of $K'$ in Theorem \ref{thm:3}, we can get the following theorem in the same way as the proof of Theorem \ref{thm:3}
\begin{theorem}\label{thm:4}
Let $\mathcal {B}$ be a block complex on a given simplicial complex $K$. Then for any abelian group $G$
and all $p,q\in \mathbb{Z}$,
\begin{eqnarray*}
&\mathcal {H}^{p,\,q}(\mathcal {B};G)\cong\mathcal {H}^{p,\,q}(K;G),\quad \w {\mathcal {H}}^{p,\,q}(\mathcal {B};G)\cong\w {\mathcal {H}}^{p,\,q}(K;G),\\
&\mathcal {H}_{p,\,q}(\mathcal {B};G)\cong\mathcal {H}_{p,\,q}(K;G),\quad \w {\mathcal {H}}_{p,\,q}(\mathcal {B};G)\cong\w {\mathcal {H}}_{p,\,q}(K;G).
\end{eqnarray*}
\end{theorem}
The significance of Theorem \ref{thm:4} is for simplifying the calculation of $\mathcal {H}^{*,*}(K;G)$, since generally $N^{*,*}(\mathcal {B})$ has
much less generators than $N^{*,*}(K)$.

Another application of Theorem \ref{thm:4} is to calculate the (co)homeology groups of the product space $X_1\times X_2$ of two polyhedra
$X_1$ and $X_2$.

\begin{cons}
The fact that the product of two simplices is not a simplex causes some problems with triangulating the products of spaces. However, there is a canonical triangulation of the product of two polyhedra for each choice
of orderings of their vertices. Suppose $K_1$, $K_2$ are simplicial complexes on $[m_1]$
and $[m_2]$ respectively ($[m]$ is the index set $\{1,2,\dots,m\}$). Then we construct a new simplicial complex on $[m_1]\times[m_2]$,
which we call the \emph{Cartesian product} of $K_1$ and $K_2$ and denote $K_1\times K_2$, as follows.
\begin{align*}
K_1\times K_2:=&\{\sigma\subset\sigma_1\times\sigma_2\mid \sigma_1\in K_1,\,\sigma_2\in K_2,\\
&\text{ and } i\leqslant i'\text{ implies } j\leqslant j'\text{ for any two pairs }(i,j), (i',j')\in \sigma\}.
\end{align*}
The polyhedron $|K_1\times K_2|$ defines a canonical triangulation of $|K_1|\times|K_2|$.
\end{cons}
It is easy to see that there is a canonical block complex $\mathcal {B}_{K_1\times K_2}=\{b_{\sigma_1,\sigma_2}\}$ on $K_1\times K_2$, where
$b_{\sigma_1,\sigma_2}=(K_1\times K_2)_{\sigma_1\times\sigma_2}$ (the full subcomplex of $K_1\times K_2$ on $\sigma_1\times\sigma_2$). The orientation of
$|b_{\sigma_1,\sigma_2}|=\langle\sigma_1\rangle\times\langle\sigma_2\rangle$ is induced by the orientations of
$\langle\sigma_1 \rangle$ and $\langle\sigma_2 \rangle$. The map
\[(\sigma_1\otimes\tau_1)\otimes(\sigma_2\otimes\tau_2)\mapsto (-1)^{|\tau_1|\cdot|\sigma_2|}(b_{\sigma_1,\sigma_2})\otimes(b_{\tau_1,\tau_2})\]
 gives the following isomorphism of double complexes
\[(N^{*,*}(K),\Delta)\otimes(N^{*,*}(L),\Delta)\cong (N^{*,*}(\mathcal {B}_{K\times L}),\Delta).\]
From Theorem \ref{thm:4} and K\"{u}nneth theorem, we get the following theorem immediately.
\begin{theorem}\label{thm:6}
If one of $E_{\,1}^{*,*}(K)$, $E_{\,1}^{*,*}(L)$ is torsion free and one of $\mathcal {H}^{*,*}(K)$, $\mathcal {H}^{*,*}( L)$ is torsion free, then
    \[\mathcal {H}^{*,*}(K\times L)\cong \mathcal {H}^{*,*}(K)\otimes\mathcal {H}^{*,*}( L).\]
\end{theorem}
The homoelogy version of Theorem \ref{thm:6} is obtained in the same way.

\section{Homeotopy and homeotopy type}\label{sec:3}
\begin{definition}
Let $K,L$ be two simplicial complexes. A simplicial map $f\colon K\to L$ is said to be \emph{non-degenerate} if for any simplex $\sigma\in K$,
$|f(\sigma)|=|\sigma|$.
\end{definition}

It is obvious that identity maps, inclusion maps and composite maps of non-degenerate maps are all non-degenerate.
\begin{theorem}\label{thm:2}
Let $f\colon K\to L$ be a non-degenerate map. Then $f$ induces homomorphisms of homeology groups:
\[f_*\colon \mathcal {H}_{*,*}(K;G)\to \mathcal {H}_{*,*}(L;G),\quad  f_*\colon \w {\mathcal {H}}_{*,*}(K;G)\to \w {\mathcal {H}}_{*,*}(L;G),\]

Dually, $f$ induces homomorphisms of cohomeology groups:
\[f^*\colon \mathcal {H}^{*,*}(L;G)\to \mathcal {H}^{*,*}(K;G),\quad  f^*\colon \w {\mathcal {H}}^{*,*}(L;G)\to \w {\mathcal {H}}^{*,*}(K;G).\]
\end{theorem}

\begin{proof} Note that $f$ induces homomorphisms of double complexes:
\[f_*\colon (N_{*,*}(K;G), D)\to (N_{*,*}(L;G), D), \quad f_*\colon (\w N_{*,*}(K;G),D)\to (\w N_{*,*}(L;G),D),\]
defined by  $f_*(\sigma\otimes\tau)=f(\sigma)\otimes f(\tau)$.

Dually, $f$ induces double complex homomorphisms:
\[f^*\colon (N^{*,*}(L;G), \Delta)\to (N^{*,*}(K;G), \Delta), \quad \w f^*\colon (\w N^{*,*}(L;G),\Delta)\to (\w N^{*,*}(K;G),\Delta),\]
define by $f^*(\sigma'\otimes\tau')=\sum \sigma\otimes\tau$, the summation taken over all $\sigma\otimes\tau\in N^{*,*}(K)$ such that
$f(\sigma)=\sigma'$, $f(\tau)=\tau'$. Then  the conclusions in the theorem can be readily verified.
\end{proof}
The following two basic properties of induced homomorphisms of (co)homeology groups are easily verified:

1) $(fg)_*=f_*g_*$, $(fg)^*=g^*f^*$ for a composed map $K_0\xrightarrow{g}K_1\xrightarrow{f}K_2$.

2) $\mathbbold{1}_*=\mathbbold{1}$ where $\mathbbold{1}$ denotes the identity map of a space or a group.

So the correspondence $K\mapsto  \mathcal {H}_{*,*}(K;G)$ (resp. $\mathcal {H}^{*,*}(K;G)$) gives rise to a  covariant (resp. contravariant) functor from
the category of simplicial complexes and non-degenerate maps to the category of groups  and group homomorphisms.
\begin{definition}\label{def:1}
Two non-degenerate maps  $f_0,f_1\colon K\to L$ are said to be \emph{homeotopic}, denoted by $f_0\approxeq f_1$,
if there is a non-degenerate map $H\colon I\times K\to I\times L$ ($I$ is any triangulation of $[0,1]$) such that $H|_{(i,K)}=f_i$ for $i=0,1$. $H$ is called a \emph{homeotopy (map) from $f_0$ to $f_1$}.
\end{definition}
\begin{theorem}\label{thx}
If two non-degenerate maps $f,g\colon K\to L$ are homeotopic, then they induce the same (co)homeology group homomorphism.
\end{theorem}
\begin{proof} We only prove the cohomeology case.
First we define $1$-dimensional simplicial complex $I_n$ ($n{>}0$) as follows. The vertex set of $I_n$ is $\{v_i\,|\,i=0,1,\dots,n\}$,
and the edge set is
\[\{e_i=\{v_{i-1},v_{i}\}\,|\,i=1,\dots,n\}.\] Then $\{I_n\}$ is the set of triangulations of $I$.
From Example \ref{exmp:1} we have $\mathcal {H}^{p,\,q}(I_n)=0$ if $(p,\,q)\neq(1,1)$ and
$\mathcal {H}^{1,1}(I_n)=\Bbb Z$. It is easily verified that $\alpha=\sum_{i=1}^{n}e_i\otimes e_i$ is a generator of
$\mathcal {H}^{1,1}(I_n)$. This implies that any
non-degenerate map from $I_m$ to $I_n$ induces isomorphism of cohomeology groups. So by theorem \ref{thm:6},
for any simplicial complex $K$ and any non-degenerate map $\varphi\colon I_m\to I_n$, the map $\varphi\times\mathbbold{1}\colon I_m\times K\to I_n\times K$ induces isomorphism of cohomeology groups
\[(\varphi\times\mathbbold{1})^*:\mathcal {H}^{*,*}(I_n\times K)\to \mathcal {H}^{*,*}(I_m\times K).\]

Given a homeotopy $H\colon I_n\times K\to I_n\times L$ from $f$ to $g$, there is a commutative diagram:
\[
\begin{CD}
I_1\times K @>\mathbbold{1}\times f>> I_1\times L \\
@V{\varphi_0\times \mathbbold{1}_K}VV @V{\varphi_0\times \mathbbold{1}_L}VV\\
I_{n+2} \times K @>\Phi>> I_{n+2}\times L \\
@A{\varphi_1\times \mathbbold{1}_K}AA @A{\varphi_1\times \mathbbold{1}_L}AA\\
I_1\times K @>\mathbbold{1}\times g>> I_1\times L
\end{CD}
\]
where $\varphi_0$ and $\varphi_1$ are simplicial maps so that $\varphi_0(e_1)=e_1$, $\varphi_1(e_1)=e_{n+2}$, and $\Phi$ is defined as follows:
\[
\Phi|_{e_i\times K}=
\begin{cases}
\varphi_0\times f&\ i=1,\\
H&\ 2\leqslant 1\leqslant n+1,\\
\varphi_1\times g&\ i=n+2.
\end{cases}\]
By the above analysis, all the vertical maps induce identity isomorphisms of cohomeology groups.
Thus $\Phi^*=(\mathbbold{1}\times f)^*=(\mathbbold{1}\times g)^*$. Applying Theorem \ref{thm:6} again, we obtain $f^*=g^*$.
\end{proof}
\begin{definition} Two polyhedra $X$ and $Y$ are said to be \emph{homeotopic equivalent}, or to have the same \emph{homeotopy type}, denoted by $X\approxeq Y$, if there are non-degenerate maps $f\colon X\to Y$ and $g\colon Y\to X$ such that $gf\approxeq 1_X$ and $fg\approxeq 1_Y$. $f$ is called a \emph{homeotopy equivalence from $X$ to $Y$}.
\end{definition}
Homeotopy type is a coarser relation than $PL$ homeomorphism. For example, the following two $2$-dimensional complexes are of the same homeotopy type but not $PL$ homeomorphic.\vspace{3mm}

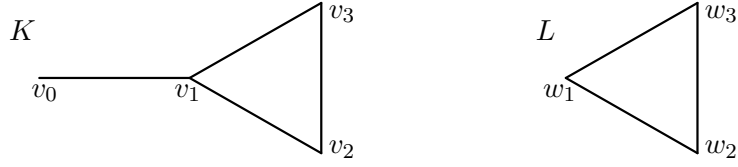
\begin{center}
\setlength{\unitlength}{0.5mm}
\begin{picture}(200,40)(0,0)
\thicklines
\drawline(0,20)(40,20)
\drawline(40,20)(75,0)
\drawline(40,20)(75,40)
\drawline(75,0)(75,40)

\drawline(140,20)(175,0)
\drawline(140,20)(175,40)
\drawline(175,0)(175,40)

\put(-8,30){{$K$}}
\put(132,30){{$L$}}
\put(-2,15){{$v_0$}}
\put(36,15){{$v_1$}}
\put(77,0){{$v_2$}}
\put(77,36){{$v_3$}}
\put(134,15){{$w_1$}}
\put(177,0){{$w_2$}}
\put(177,36){{$w_3$}}
\end{picture}
\end{center}\vspace{3mm}

$f\colon K\to L$ is defined by $f(v_i)=w_i$ for $i=1,2,3$ and $f(v_0)=w_3$. $g\colon L\to K$ is the inclusion $g(w_i)=v_i$ for $i=1,2,3$. The homeotopy from $\mathbbold{1}_K$ to $gf$ is shown in the following picture, which is identity on the triangular prism and maps the left two $2$-simplices $1$ and $2$ respectively to the right two simplices $1$ and $2$.\vspace{3mm}

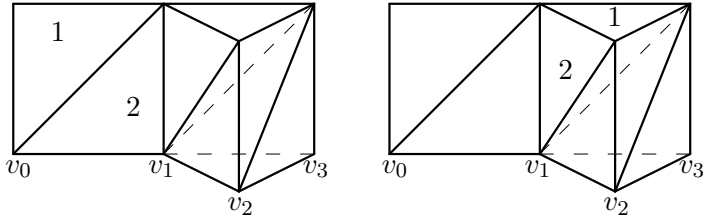
\begin{center}
\setlength{\unitlength}{0.5mm}
\begin{picture}(200,40)(0,0)
\thicklines
\drawline(0,0)(0,40)
\drawline(0,0)(40,0)
\drawline(0,40)(40,40)
\drawline(40,0)(40,40)
\drawline(0,0)(40,40)
\drawline(40,0)(60,-10)
\drawline(60,-10)(80,0)
\drawline(40,40)(60,30)
\drawline(60,30)(80,40)
\drawline(80,0)(80,40)
\drawline(40,0)(60,30)
\drawline(60,-10)(80,40)
\drawline(40,40)(80,40)
\drawline(60,-10)(60,30)
\thinlines
\dashline{4}(40,0)(80,0)
\dashline{4}(40,0)(80,40)

\put(10,30){1}
\put(30,10){2}

\thicklines
\drawline(100,0)(100,40)
\drawline(100,0)(140,0)
\drawline(100,40)(140,40)
\drawline(140,0)(140,40)
\drawline(100,0)(140,40)
\drawline(140,0)(160,-10)
\drawline(160,-10)(180,0)
\drawline(140,40)(160,30)
\drawline(160,30)(180,40)
\drawline(180,0)(180,40)
\drawline(140,0)(160,30)
\drawline(160,-10)(180,40)
\drawline(140,40)(180,40)
\drawline(160,-10)(160,30)
\thinlines
\dashline{4}(140,0)(180,0)
\dashline{4}(140,0)(180,40)

\put(158,34){1}
\put(145,20){2}

\put(-2,-5){{$v_0$}}
\put(36,-5){{$v_1$}}
\put(57,-15){{$v_2$}}
\put(77,-5){{$v_3$}}

\put(98,-5){{$v_0$}}
\put(136,-5){{$v_1$}}
\put(157,-15){{$v_2$}}
\put(177,-5){{$v_3$}}
\end{picture}
\end{center}
\vspace{10mm}

Homeotopy type is finer than homotopy type. For example, $D^n$ and $D^m$ ($m\neq n$) are not of the same homeotopy type since they have
different (co)homeology groups. But they are of the same homotopy type. Another example is $C_3S^n$ and $S^{n+1}\vee S^{n+1}$
(see the examples at the end of \S \ref{sec:2}).

\section{The geometric description of (co)homeology groups}\label{sec:4}

\begin{definition}\label{def:2}
A simplicial complex $K$ is said to be \emph{completely connected} if for any two simplices $\sigma, \sigma'$ having the same dimension, there is a sequence $\sigma_0,\sigma_1,\dots,\sigma_m$ of simplices of $K$ such that $\sigma_0=\sigma,\ \sigma_m=\sigma'$, $|\sigma_{i}|=|\sigma|$  and $\sigma_i$, $\sigma_{i+1}$ share a common facet for each $i$. A \emph{completely connected component of $K$} is a maximal completely connected subcomplex of $K$.
\end{definition}
Note that two different completely connected components may have nonempty intersection. For example, if $K=\Delta^2\vee\Delta^2$, then $K$ has two
completely connected components $K_1=\Delta^2\vee\partial\Delta^2$ and $K_2=\partial\Delta^2\vee\Delta^2$,
and $K_1\cap K_2=\partial\Delta^2\vee\partial\Delta^2$.
\begin{lemma}\label{lem:2}
If $K_1$ and $K_2$ are two different completely connected components of $K$, then
\[\mathrm{dim}(K_1\cap K_2)<\mathrm{min}\,\{\mathrm{dim}\,K_1,\, \mathrm{dim}\,K_2\}.\]
\end{lemma}
\begin{proof}
Suppose on the contrary that $\mathrm{dim}(K_1\cap K_2)=\mathrm{min}\,\{\mathrm{dim}\,K_1,\, \mathrm{dim}\,K_2\}$, and without loss of generality
$\mathrm{dim}(K_1\cap K_2)=\mathrm{dim}\,K_1$. Then by the definition of completely connected component, any simplex $\sigma\in K_1$ can be connected to a simplex $\tau\in K_1\cap K_2$ having the same dimension with $\sigma$, and so to a simplex $\tau'\in K_2$ with the same dimension. This implies that $K_1\subset K_2$, but from the maximality of completely connected components we obtain $K_1=K_2$, a contradiction.
\end{proof}
\begin{theorem}\label{th6}
Let $K$ be a simplicial complex. Then $\mathcal {H}^{n,n}(K)$ is free for $n\geqslant0$ and {\rm rank}\,$\mathcal {H}^{n,n}(K)$
equals the number of $n$-dimensional completely connected components of $K$.
\end{theorem}
\begin{proof}
Since $E^{n,n}_{\,1}(K)=\bigoplus_{|\sigma|=n+1}\w H^{-1}(\text{link}_K \sigma)$ is free for $n\geqslant 0$ and
$\mathcal {H}^{n,n}(K)=\text{Ker}\,\Delta_1^{n,n}$,
then we deduce the first assertion from the fact that every subgroup of a free abelian group is free.

To prove the second assertion, suppose $\{K_i\}$ is the set of completely connected components of $K$.  Let $K'=\bigcup_{\mathrm{dim}K_i>n}K_i$,
then there is a commutative diagram:
\[\begin{CD}
E^{n,n}_1(K)@>\cong>>E^{n,n}_1(K')\oplus\big(\bigoplus_{\mathrm{dim}K_i=n}E^{n,n}_1(K_i)\big)\\
@V\Delta_1VV @V\oplus\Delta_1VV\\
E^{n-1,n}_1(K)@>\cong>>E^{n-1,n}_1(K')\oplus\big(\bigoplus_{\mathrm{dim}K_i=n}E^{n-1,n}_1(K_i)\big)
\end{CD}\]
where the horizontal isomorphisms come from Lemma \ref{lem:1} and Lemma \ref{lem:2}. Thus we have
\[\mathcal {H}^{n,n}(K)\cong \mathcal {H}^{n,n}(K')\oplus\big(\bigoplus_{\mathrm{dim}K_i=n}\mathcal {H}^{n,n}(K_i)\big)\]
If we prove that $\mathcal {H}^{n,n}(K')=0$ and $\mathcal {H}^{n,n}(K_i)=\mathbb{Z}$ for each $K_i$ satisfying $\mathrm{dim}K_i=n$,
then the theorem holds.

First we prove $\mathcal {H}^{n,n}(K_i)=\mathbb{Z}$ if $\mathrm{dim}K_i=n$. Suppose $x\in \mathcal {H}^{n,n}(K_i)$, then a representative $\xi$ of $x$ in $F^n/F^{n-1}$ ($F^p=\bigoplus_{j\leqslant p}N^{j,*}(K_i)$) has the form \[\xi=\sum_{|\sigma|=n+1}n_\sigma\cdot(\sigma\otimes\sigma),\quad n_\sigma\in \mathbb{Z}.\]
By definition, for arbitrary two $n$-simplices $\sigma',\sigma''\in K_i$, there is a sequence $\sigma'=\sigma_0,\sigma_1,\dots,\sigma_m=\sigma''$ of $n$-simplices
such that $\sigma_j, \sigma_{j+1}$ share a common $(n-1)$-simplex $\tau_j$ for $0\leqslant j<m$. Let \[p_{\tau_i}\colon E^{n-1,n}_1(K_i)=\bigoplus_{|\tau|=n}\w H^0(\text{link}_{K_i} \tau)\to\w H^0(\text{link}_{K_i} \tau_j)\]
be the projection map.
Then $p_{\tau_1}\Delta_1(\xi)=\sum_\sigma n_\sigma\cdot(\tau_1\otimes\sigma)=0$, where the summation is taken over all $n$-chain simplices $\sigma$
such that $\sigma=\tau_1*(v)$ for some vertex $v$. Since  in $\w H^0(\text{link}_{K_i}\tau_1)$ the zero class is represented by
$k\cdot\sum_{v}(v)$, where $k\in\mathbb{Z}$ and the summation is taken over all vertices of $\text{link}_{K_i}\tau_1$,
then from the correspondence $\tau_1\otimes(\tau_1*(v))\mapsto(v)$ given in the proof of Lemma \ref{lem:1} we have $n_{\sigma_1}=n_{\sigma_2}$. Repeating this process gives $n_{\sigma'}=n_{\sigma''}$.
From the arbitrariness of $\sigma'$ and $\sigma''$, we may rewrite $\xi$ as
\[\xi=k\cdot\sum_{|\sigma|=n+1}\sigma\otimes\sigma,\quad k\in\mathbb{Z}.\]
On the other hand, since $\mathrm{dim}K_i=n$, then $\xi_0=\sum_{|\sigma|=n+1}\sigma\otimes\sigma$ is a cocycle (evidently not a coboundary) of $F^{p}/F^{p-1}$. Thus $[\xi_0]\in E^{n,n}_1(K_i)$. Notice that $p_\tau\Delta_1([\xi_0])=0$ for all $(n-1)$-simplices $\tau\in K_i$,
then from the fact that
$\Delta_1=\bigoplus_{|\tau|=n}p_\tau\Delta_1$ we have $\Delta_1([\xi_0])=0$. Hence $\mathcal {H}^{n,n}(K_i)=\mathbb{Z}$ with a generator $[\xi_0]$.

Now we prove $\mathcal {H}^{n,n}(K')=0$. Suppose $x\in \mathcal {H}^{n,n}(K')$, then as above $x$ has a representative  \[\xi=\sum_{|\sigma|=n+1}n_\sigma\cdot(\sigma\otimes\sigma),\quad n_\sigma\in \mathbb{Z},\] where $n_\sigma=0$ if $\text{link}_{K'}\sigma\neq\varnothing$. For any $n$-simplex $\sigma'\in K_j$ with $\mathrm{dim}K_j>n$, there is a sequence $\sigma'=\sigma_0,\sigma_1,\dots,\sigma_m=\sigma''$ of $n$-simplices of $K_j$ such that
$\sigma''\in \partial\tau$ for a $(n+1)$-simplex $\tau\in K_j$.
By the preceding analysis we have $n_{\sigma'}=n_{\sigma''}$. However, $\text{link}_{K'}\sigma''\neq\varnothing$, so $n_{\sigma''}=0$. From the arbitrariness of $\sigma'$ and the construction of $K'$, we have $\xi=0$, and then $\mathcal {H}^{n,n}(K')=0$.
\end{proof}

\begin{lemma}\label{le1}
Let $K$ be a simplicial complex. Then:
\begin{enumerate}[(a)]

\item $\mathcal {H}^{*,\,q}(K)\cong \mathcal {H}^{*,\,q}(K^{(n)})$ for $0\leqslant q<n$.

\item $E^{p,\,p}_{\,r}(K)\cong E^{p,\,p}_{\,r}(K^{(n)})$ for $r>1$ and $0\leqslant p<n$;\\
$E^{p,\,q}_{r}(K)\cong E^{p,\,q}_{r}(K^{(n)})$ for $r>1$ and
$0\leqslant p<q\leqslant n-(r-1)(q-p)$.
\end{enumerate}

These isomorphisms are all induced by corresponding inclusions (e.g., (a) is induced by the inclusion $K^{(n)}\hookrightarrow K$).
\end{lemma}

\begin{proof}
(a) According to Lemma \ref{lem:1}, $E_{\,1}^{p,\,q}(K)=\bigoplus_{|\sigma|=p+1}\,\w H^{q-p-1}({\rm link}_K\sigma)$. An easy observation shows that the inclusion $K^{(n)}\hookrightarrow K$ induces isomorphisms \[\w H^{q-p-1}(\text{link}_K\sigma)\cong\w H^{q-p-1}(\text{link}_{K^{(n)}}\sigma)\] for all $\sigma\in K$ with $|\sigma|=p+1\leqslant q+1\leqslant n$.
Thus the fact that $\mathcal {H}^{*,\,q}(K)$ is the homology of $(E_{\,1}^{*,\,q}(K),\,\Delta_1)$ gives the desired formula.

(b) We prove this by induction on $r$, starting with the case $r=2$ which is just a special case of (a). Assuming inductively there are two isomorphisms
$E^{p,\,p}_{r}(K)\cong E^{p,\,p}_{r}(K^{(n)})$ and $E^{p-r,\,p-r+1}_{r}(K)\cong E^{p-r,\,p-r+1}_{r}(K^{(n)})$ for $p<n$ induced by the inclusion $K^{(n)}\hookrightarrow K$, then we obtain $E^{p,\,p}_{r+1}(K)\cong E^{p,\,p}_{r+1}(K^{(n)})$ from the fact $E^{p,\,p}_{r+1}(K)=\mathrm{Ker}\Delta_r^{p,\,p}$.
To get the second isomorphism consider the following commutative diagram induced by  $K^{(n)}\hookrightarrow K$ for $p<q\leqslant n-(r-1)(q-p)$
\[
\begin{CD}
E^{p+r,\,q+r-1}_{r}(K)@>\Delta_{r}>>E^{p,\,q}_{r}(K)@>\Delta_{r}>>E^{p-r,\,q-r+1}_{r}(K)\\
@VV\cong V @VV\cong V @VV\cong V\\
E^{p+r,\,q+r-1}_{r}(K^{(n)})@>\Delta_{r}>>E^{p,\,q}_{r}(K^{(n)})@>\Delta_{r-1}>>E^{p-r,\,q-r+1}_{r}(K^{(n)})
\end{CD}
\]
where the vertical isomorphisms come from induction.
Then the fact $E_{r+1}=H^*(E_r,\Delta_r)$ implies that $E^{p,\,q}_{r+1}(K)\cong E^{p,\,q}_{r+1}(K^{(n)})$, finishing the induction step.
\end{proof}
\begin{lemma}\label{lem:3}
Let $K$ be an $n$ dimensional complex with $m$ path-components, then
\[
E_\infty^{p,\,p}(K)=
\begin{cases}
\mathbb{Z}^m&\quad \text{ if } p=n,\\
0&\quad \text{ otherwise. }
\end{cases}
\]
\end{lemma}
\begin{proof}
First we do the case $K$ is path-connected. Set \[\lambda=\sum_{\sigma\in K}(-1)^{[\frac{|\sigma|-1}{2}]}\sigma\otimes\sigma\in N^{*,*}(K),\] where $[\cdot]$ denotes the integer part. Then
\begin{align*}\Delta(\lambda)&=\sum_{\sigma\in K}(-1)^{[\frac{|\sigma|-1}{2}]}d\otimes1(\sigma\otimes\sigma)+\sum_{\sigma\in K}(-1)^{|\sigma|+[\frac{|\sigma|-1}{2}]}1\otimes \delta(\sigma\otimes\sigma)\\
&=\sum_{\substack{\sigma\in K,\\(v)\in K}}(-1)^{[\frac{|\sigma|}{2}]}\sigma\otimes(v)*\sigma+\sum_{\substack{\sigma\in K,\\(v)\in K}}(-1)^{|\sigma|+[\frac{|\sigma|-1}{2}]}\sigma\otimes(v)*\sigma
\end{align*}
Since $[\frac{|\sigma|-1}{2}]+[\frac{|\sigma|}{2}]=|\sigma|-1$, then $\Delta(\lambda)=0$. On the other hand, $\lambda$ is evidently not a cocycle of $(N^{*,*}(K),\Delta)$. Thus $\lambda$ is a generator of
$H^0\big(T^*(N(K))\big)=H^0(K)=\mathbb{Z}$, and so $\lambda$ must survive to $E_\infty^{*,*}(K)$.
If dim$K=n$, then $\lambda$ represents an element of $E_r^{n,n}(K)$ for $r\geqslant 1$. The statement of the lemma follows from the fact that
\[H^0\big(T^*(N(K))\big)=\bigoplus_{p\geqslant0}E_\infty^{p,p}(K).\]

The general case follows by Theorem \ref{thm:5} (a).
\end{proof}
\begin{theorem}\label{th8}
Let $K$ be a simplicial complex, $\{K_i\}$ be the set of completely connected components of $K$, $L_n=\bigcup_{\mathrm{dim}K_i>n}K_i$. Then:
\begin{enumerate}[(a)]
\item $\mathcal {H}^{n-1,n}(K)$ is free for $n\geqslant1$.
\item $\mathcal {H}^{n-1,n}(K)\cong \mathcal {H}^{n-1,n}(L_n)\oplus\big(\bigoplus_{\mathrm{dim}K_i=n}\mathcal {H}^{n-1,n}(K_i)\big)$.
\item If $L_n$ has $m_1$  completely connected components and $L^{(n)}_n$ has $m_2$ completely connected components,
then \[{\rm rank}\,\mathcal {H}^{n-1,n}(L_n)\geqslant m_1-m_2.\]
\end{enumerate}
\end{theorem}
Before proving Theorem \ref{th8}, we need the following lemma.
\begin{lemma}\label{le2}
Let $\mathcal {G}$ be a graph with vertex set $S$ (i.e. $\mathcal {G}$ is a simplicial complex of dimension $1$). If $S'$ is a subset of $S$, let $C^0(S')$ be the subgroup of $C^0(\mathcal {G})$ with basis $S'$. Then $\delta(C^0(S'))$ is a direct summand of $C^1(\mathcal {G})$.
\end{lemma}

\begin{proof} Without loss of generality we  assume $\mathcal {G}$ is path-connected. In case $S'=S$, since $H^1(\mathcal {G})$ is always torsion free, the short exact sequence \[0\to \delta(C^0(\mathcal {G}))\to C^1(\mathcal {G})\to H^1(\mathcal {G})\to 0\] splits. Then the statement follows.

For the case $S'\subsetneq S$, we need only prove that $\delta(C^0(S'))$ is a direct summand of $\delta(C^0(\mathcal {G}))$. Suppose on the contrary that $\delta(C^0(\mathcal {G}))/\delta(C^0(S'))$ has a torsion subgroup such as $\mathbb{Z}_p$ ($p>0$). Thus there exists an element \[\alpha=\sum_{v\in S}a_v\cdot(v)\in C^0(\mathcal {G}),\quad a_v\in \mathbb{Z}\] and an element \[\beta=\sum_{v\in S'} b_v\cdot(v)\in C^0(S'),\quad b_v\in \mathbb{Z}\] such that $p\cdot\delta(\alpha)=\delta(\beta)$ and $\delta(\alpha)\notin \delta(C^0(S'))$. Since $H^0(\mathcal {G})=\mathbb{Z}$ (by assumption) is generated by $\sum_{v\in S}(v)$, we have $p\cdot\alpha-\beta=k\cdot\sum_{v\in S}(v)$ for some $k>0$. Since $\delta(\alpha)\notin \delta(C^0(S'))$, there must be a vertex $u\notin S'$, such that $a_u\neq 0$, then $p\cdot a_u=k$, and so $p\mid k$. For any vertex $v\in S',\ b_v=p\cdot a_v-k$. Thus $p\mid b_v$, and therefore we have $\beta=p\cdot \beta'$ for some $\beta'\in C^0(S')$. Since $C^1(\mathcal {G})$ is free, then  $p\cdot\delta(\alpha-\beta')=\delta(p\cdot\alpha-\beta)=0$ implies that $\delta(\alpha)=\delta(\beta')$, which is a contradiction.
\end{proof}

\begin{proof}[Proof of Theorem \ref{th8}]
(a)  We use a topological technique by applying lemma \ref{le2}. First we construct a graph $\mathcal {G}^n(K)$ for $n\geqslant0$ as follows. For each generator $\sigma\otimes\sigma$ of $E^{n,n}_1(K)$ (i.e., $\sigma$ is an $n$ dimensional facet of $K$), give a vertex $v_\sigma$ corresponding to $\sigma$. If $\sigma _i$ is an $(n-1)$-face of $\sigma$
($\{v_i\}=\sigma\setminus\sigma_i$) such that $v_i$ is a path-component of $\text{link}_K\sigma$, then
give a vertex $v_{\sigma_i}$ corresponding to $\sigma_i$. $\mathcal {G}^n(K)$ is defined to be a graph having all such $v_\sigma$ and $v_{\sigma_i}$ as vertices and all $(v_{\sigma_i},v_\sigma)$ as edges. Let $S$ be the vertex subset consisting of all $v_\sigma$ with $|\sigma|=n+1$. Then we have a commutative diagram:
\[
\xymatrix{C^0(S)\ar[d]^{\delta}\ar[rr]^-{\eta_0}_-\cong&&\underset{|\sigma|=n+1}{\bigoplus}\w H^{-1}(\mathrm{link}_K\sigma)\cong E^{n,n}_1(K)\ar[d]^{\Delta_1^{n,n}}\\
C^1(\mathcal {G}^n(K))\ar[r]^-{\eta_1}&\underset{|\sigma|=n}{\bigoplus}H^{0}(\mathrm{link}_K\sigma)\ar[r]^-h
&\underset{|\sigma|=n}{\bigoplus}\w H^{0}(\mathrm{link}_K\sigma)\cong E^{n-1,n}_1(K)}
\]
where $h$ is the natural quotient map, $\eta_0$ is defined by $\eta_0(v_\sigma)=(\varnothing)_\sigma$ (a generator of $H^{-1}(\mathrm{link}_K\sigma)$), and $\eta_1$ is defined by $\eta_1((v_{\sigma_i},v_\sigma))=v_i$. It is clear that $\mathrm{Im}\,\eta_1$ is a direct summand of $\bigoplus_{|\sigma|=n}H^{0}(\mathrm{link}_K\sigma)$. From this commutative diagram and lemma \ref{le2} we have Im$\Delta_1^{n,n}$ is a direct summand of $ E^{n-1,n}_1(K)$,  therefore  $\mathcal {H}^{n-1,n}(K)$ as a subgroup of
$E^{n-1,n}_1(K)/\text{Im}\Delta_1^{n,n}$ is free.

(b) C.F. the analysis in the proof of Theorem \ref{th6}.

(c) According to lemma \ref{le1} (a), we may assume $L_n$ is of dimension $n+1$. Let $\{\Gamma_j\}_{1\leqslant j\leqslant m_2}$ be the set of completely connected components of $L^{(n)}_n$. For any $K_i\subset L_n$, clearly $K_i^{(n)}= \Gamma_j$ for some $j$, so $m_1\geqslant m_2$.
According to theorem \ref{th6},
\[\mathcal {H}^{n+1,n+1}(L_n)=\mathbb{Z}^{m_1}\ \text{ and }\ \mathcal {H}^{n,n}(L^{(n)}_n)=\mathbb{Z}^{m_2}.\]

Start with the special case $m_2=1$. We claim that the differentials
\[\Delta_r\colon E_r^{n+1,n+1}(L_n)\to E_r^{n+1-r,n+2-r}(L_n)\] vanish for all $r>2$.
Suppose on the contrary that $\Delta_r([c])\neq 0$ for some cohomology class $[c]\in E_r^{n+1,n+1}(K)$ and some $r>2$. If
$c=\sum_{\sigma\in L_n} n_\sigma\cdot(\sigma\otimes\sigma)$ ($n_\sigma\in\mathbb{Z}$) is a representative of $[c]$ in
$N^{*,*}(L)$ (clearly, there is at least one $(n+1)$-simplex $\sigma$ such that $n_\sigma\neq0$), then by assumption \[\Delta (c)\in F^{n+1-r}(L_n)=\bigoplus_{i\leqslant n+1-r}N^{i,*}(L_n),\] and $[\Delta(c)]\neq0$ in $E_t^{n+1-r,n+2-r}(L_n)$ for $t\leqslant r$.

Let $f$ be the inclusion $L^{(n)}_n\hookrightarrow L_n$ and let $b=f^*(c)$ (see the proof of Theorem \ref{thm:2} for the definition of $f^*$).
Since $\Delta (c)\in F^{n+1-r}(L_n)$ and $r>2$, there must be an $n$-simplex $\tau\in L_n$ such that $n_\tau\neq0$, then $b\in F^{n}(L^{(n)}_n)$, $b\notin F^{n-1}(L^{(n)}_n)$, and $\Delta(b)=f^*(\Delta(c))\in F^{n+1-r}(L^{(n)}_n)$,
thus $[b]$ is an element of $E_{r-1}^{n,n}(L^{(n)}_n)$. Since \[\Delta_{r-1}([b])=[\Delta(b)]=[f^*(\Delta(c))]=f^*([\Delta(c)]),\] and by Lemma \ref{le1} (b) \[f^*:E^{n+1-r,n+2-r}_{r-1}(L^{(n)}_n)\cong E^{n+1-r,n+2-r}_{r-1}(L_n),\] then $\Delta_{r-1}([b])\neq0$.
But this is contrary to the fact that
$\mathcal {H}^{n,n}(L^{(n)}_n)=\mathbb{Z}$ survive to $E_\infty^{n,n}(L^{(n)}_n)=\mathbb{Z}$ (Lemma \ref{lem:3}).

From this conclusion and the fact that $\mathcal {H}^{n+1,n+1}(L_n)=\mathbb{Z}^{m_1},\ E_\infty^{n+1,n+1}(L_n)=\mathbb{Z}$ (Lemma \ref{lem:3}), it follows that the image of $\Delta_2\colon \mathcal {H}^{n+1,n+1}(L_n)\to \mathcal {H}^{n-1,n}(L_n)$ must be isomorphic to $\mathbb{Z}^{m-1}$. Hence
rank\,$\mathcal {H}^{n-1,n}(L_n)\geqslant m_1-1$.

For the general case, Let $M_j=\bigcup_{K_i^{(n)}=\Gamma_j}K_i$. Then by Lemma \ref{lem:2} $\mathrm{dim}(M_j\cap M_{j'})<n$ if $j\neq j'$.
Thus by the same reasoning as in the proof of Theorem \ref{th6}, we have
\[\mathcal {H}^{n-1,n}(L_n)=\bigoplus_{j=1}^{m_2}\mathcal {H}^{n-1,n}(M_j).\]
Then the theorem follows immediately.
\end{proof}
\begin{lemma}\label{th4}
Let $K$ be a simplicial complex, then:
\[\sum_{0\leqslant p\leqslant q}(-1)^{q-p}\cdot{\rm rank}\,\mathcal {H}^{p,\,q}(K)=\chi(K),\] where $\chi(K)$ is the Euler characteristic of $K$.
\end{lemma}
\begin{proof}
Since $E_{r+1}^{*,*}(K)=H(E_{r}^{*,*}(K),\Delta_r)$, then
 \[\sum_{0\leqslant p\leqslant q}(-1)^{q-p}\cdot\textrm{rank}\,E_{r+1}^{p,\,q}(K)=\sum_{0\leqslant p\leqslant q}(-1)^{q-p}\cdot\textrm{rank}\,E_{r}^{p,\,q}(K)\]
for all $r\geqslant 1$.
On the other hand since the spectral sequence $\{E^{*,*}_r(K)\}$ converges to the cohomology group $H^*(K)$, then
\[\sum_{0\leqslant p\leqslant q}(-1)^{q-p}\cdot\textrm{rank}\,E_{\infty}^{p,\,q}(K)=\sum_{n\geqslant0}(-1)^n\cdot \text{rank}\, H^n(K)=\chi(K).\]
Then the conclusion follows.
\end{proof}
\begin{theorem}\label{th7}
Let $K, L$ be two simplicial complexes. If $K\cap L=2^\sigma$ with $\mathrm{dim}\,\sigma=n$, then we can give $\mathcal {H}^{*,*}(K\cup L)$
in terms of $\mathcal {H}^{*,*}(K)$ and $\mathcal {H}^{*,*}(L)$ by the following formulae:

(a) If $K$ (or $L$) has a completely connected component of dimension $n$ containing $\sigma$, then
\[
\mathcal {H}^{p,\,q}(K\cup L)=
\begin{cases}
\big(\mathcal {H}^{p,\,q}(K)\oplus \mathcal {H}^{p,\,q}(K)\big)/\mathbb{Z}& \ \text{for}\ p=q=n\\
\mathcal {H}^{p,\,q}(K)\oplus \mathcal {H}^{p,\,q}(K)& \ \text{otherwise}
\end{cases}
\]

(b) If the condition in (a) dose not hold but $K$ or $L$ has a completely connected component of dimension $n+1$ containing $\sigma$, and if

\quad (i) $\sigma$ is a facet in $K$ or in $L$, then
\[
\mathcal {H}^{p,\,q}(K\cup L)=
\begin{cases}
\mathcal {H}^{p,\,q}(K)\oplus \mathcal {H}^{p,\,q}(K)\oplus\mathbb{Z}& \ \text{for}\ p=n-1,\ q=n\\
\mathcal {H}^{p,\,q}(K)\oplus \mathcal {H}^{p,\,q}(K)& \ \text{otherwise}
\end{cases}
\]

\quad (ii) $\sigma$ is not a facet in both $K$ and $L$, then
\[
\mathcal {H}^{p,\,q}(K\cup L)=
\begin{cases}
\mathcal {H}^{p,\,q}(K)\oplus \mathcal {H}^{p,\,q}(K))/\mathbb{Z}& \ \text{for}\ p=q=n+1\\
\mathcal {H}^{p,\,q}(K)\oplus \mathcal {H}^{p,\,q}(K)& \ \text{otherwise}
\end{cases}
\]

(c) If the completely connected components of $K$ and $L$, which contain $\sigma$, all have dimension larger than $n+1$, and if

\quad (i) $\sigma$ is a facet of $K$ or $L$, then
\[
\mathcal {H}^{p,\,q}(K\cup L)=
\begin{cases}
\mathcal {H}^{p,\,q}(K)\oplus \mathcal {H}^{p,\,q}(K)\oplus\mathbb{Z}& \ \text{for}\ p=n-1,\ q=n\\
\mathcal {H}^{p,\,q}(K)\oplus \mathcal {H}^{p,\,q}(K)& \ \text{otherwise}
\end{cases}
\]

\quad (ii) $\sigma$ is not a facet in both $K$ and $L$, then
\[
\mathcal {H}^{p,\,q}(K\cup L)=
\begin{cases}
\mathcal {H}^{p,\,q}(K)\oplus \mathcal {H}^{p,\,q}(K)\oplus\mathbb{Z}& \ \text{for}\ p=n,\ q=n+1\\
\mathcal {H}^{p,\,q}(K)\oplus \mathcal {H}^{p,\,q}(K)& \ \text{otherwise}
\end{cases}
\]
\end{theorem}

Before proving the theorem, we need a fact that $\chi(K\cup L)=\chi(K)+\chi(L)-1$, which follows from the Mayer-Vietoris sequence for
$(K\cup L,\,K,\,L)$.

\begin{proof}[proof of Theorem \ref{th7}]
(a)\, In this case, according to Lemma \ref{lem:1}, an easy verification shows that
\[
E_1^{p,\,q}(K\cup L)=
\begin{cases}
\big(E_1^{p,\,q}(K)\oplus E_1^{p,\,q}(L)\big)/\mathbb{Z}& \ \text{for}\ p=q=n\\
E_1^{p,\,q}(K)\oplus E_1^{p,\,q}(L)& \ \text{otherwise}
\end{cases}
\]
So the only possible $p,\,q$ for which $\mathcal {H}^{p,\,q}(K\cup L)\neq \mathcal {H}^{p,\,q}(K)\oplus \mathcal {H}^{p,\,q}(K)$ are $p=q=n$ and $p=n-1,\ q=n$. Let $n_0,\,n_1$ and $n_2$ are the numbers of completely connected components of dimension $n$ of $K\cup L$, $K$ and $L$ respectively.
Then by assumption, $n_0=n_1+n_2-1$. Theorem \ref{th6} gives that
\[\mathcal {H}^{n,n}(K\cup L)=(\mathcal {H}^{n,n}(K)\oplus \mathcal {H}^{n,n}(L))/\mathbb{Z}.\]
According to Lemma \ref{th4} and the preliminary argument before this proof, we have
\begin{equation}
\begin{split}\label{eq:1}
\sum_{0\leqslant p\leqslant q}(-1)^{q-p}\cdot{\rm rank}\,\mathcal {H}^{p,\,q}(K\cup L)=\sum_{0\leqslant p\leqslant q}(-1)^{q-p}\cdot{\rm rank}\,\mathcal {H}^{p,\,q}(K)\\+\sum_{0\leqslant p\leqslant q}(-1)^{q-p}\cdot{\rm rank}\,\mathcal {H}^{p,\,q}(L)-1
\end{split}
\end{equation}
These formulae together imply that
\[\mathrm{rank}\,\mathcal {H}^{n-1,n}(K\cup L)=\mathrm{rank}\,\mathcal {H}^{n-1,n}(K)+\mathrm{rank}\,\mathcal {H}^{n-1,n}(L).\]
Then we get the desired formula by appealing to theorem \ref{th8} (a).

(b)-(i)\, As in (a) the only possible $p,\,q$ for which $\mathcal {H}^{p,\,q}(K\cup L)\neq \mathcal {H}^{p,\,q}(K)\oplus \mathcal {H}^{p,\,q}(K)$ are $p=q=n$ and $p=n-1,\ q=n$. But evidently
\[\mathcal {H}^{n,n}(K\cup L)=\mathcal {H}^{n,n}(K)\oplus \mathcal {H}^{n,n}(K)\] in this case.
So formula \eqref{eq:1} implies that
\[\text{rank}\,\mathcal {H}^{n-1,n}(K\cup L)=\text{rank}\,\mathcal {H}^{n-1,n}(K)+\,\text{rank}\,\mathcal {H}^{n-1,n}(L)+1.\]
Appealing to theorem \ref{th8} (a) again, we have \[\mathcal {H}^{n-1,n}(K\cup L)=\mathcal {H}^{n-1,n}(K)\oplus \mathcal {H}^{n-1,n}(K)\oplus\mathbb{Z},\]
and so get the desired formula.

(b)-(ii)\, In this case
\[
E_1^{p,\,q}(K\cup L)=
\begin{cases}
E_1^{p,\,q}(K)\oplus E_1^{p,\,q}(L)\oplus\mathbb{Z}& \ \text{for}\ p=n,\ q=n+1\\
E_1^{p,\,q}(K)\oplus E_1^{p,\,q}(L)& \ \text{otherwise}
\end{cases}
\]
The added $\mathbb{Z}$ summand comes from
\[\w H^0(\text{link}_{K\cup L}\sigma)=\w H^0(\text{link}_K\sigma)\oplus\w H^0(\text{link}_L\sigma)\oplus\mathbb{Z}.\]
A generator of this $\mathbb{Z}$ summand is represented by
\[\xi=\sum_{(v)\in \text{link}_K\sigma}\sigma\otimes\sigma*(v).\]
For any $(n-1)$-face $\sigma_i$ of $\sigma$, let $p_{\sigma_i}:N^{*,*}\to N^{\sigma_i,*}$ be the projection. Then 
\[p_{\sigma_i}\comp (d\otimes1)(\xi)=\pm \sum_{(v)\in \text{link}_K\sigma}\sigma_i\otimes\sigma*(v)=\pm (1\otimes\delta)(\sigma_i\otimes\sigma).\]
It follows that $\Delta_1([\xi])=0$. Thus the only possible $p,\,q$ for which $\mathcal {H}^{p,\,q}(K\cup L)\neq \mathcal {H}^{p,\,q}(K)\oplus \mathcal {H}^{p,\,q}(K)$ are $p=q=n+1$ and $p=n,\ q=n+1$. Let $k_0,\,k_1$ and $k_2$ are the numbers of completely connected components of dimension $n+1$ of $K\cup L$, $K$ and $L$ respectively. It is easy to see that $k_0=k_1+k_2-1$ in this case. Thus Theorem \ref{th6} gives
that \[\mathcal {H}^{n+1,n+1}(K\cup L)=(\mathcal {H}^{n+1,n+1}(K)\oplus \mathcal {H}^{n+1,n+1}(K))/\mathbb{Z},\]
and so the desired formula follows.

(c)-(i)\, The proof is completely identical with (b)-(i).

(c)-(ii)\, The proof is similar with (b)-(ii). we need only notice that
\[\mathcal {H}^{n+1,n+1}(K\cup L)=\mathcal {H}^{n+1,n+1}(K)\oplus \mathcal {H}^{n+1,n+1}(K)\] in this case.
\end{proof}
\providecommand{\bysame}{\leavevmode\hbox to3em{\hrulefill}\thinspace}
\providecommand{\MR}{\relax\ifhmode\unskip\space\fi MR }
\providecommand{\MRhref}[2]{%
  \href{http://www.ams.org/mathscinet-getitem?mr=#1}{#2}
}
\providecommand{\href}[2]{#2}


\begin{thebibliography}{10}

\bibitem{A30}
J.~W. Alexander, \emph{The combinatorial theorey of complexes}, Ann. of Math.
  \textbf{31} (1930), no.~2, 292--320.

\bibitem{BH98}
W.~Bruns and J.~Herzog, \emph{Cohen-macaulay rings}, revised ed., Cambridge
  Studies in Adv. Math., vol.~39, Cambridge Univ. Press, Cambridge, 1998.

\bibitem{G67}
B.~Gr\"unbaum, \emph{Convex polytopes}, 2nd ed., Graduate Texts in Math., vol.
  221, Springer-Verlag, New York, 1967.

\bibitem{H69}
J.~F.~P. Hudson, \emph{Piecewise linear topology}, W. A. Benjamin, Inc, New
  York-Amsterdam, 1969.

\bibitem{L99}
W.~B.~R. Lickorish, \emph{Simplicial moves on complexes and manifolds}, Geom.
  Topol. Monogr. \textbf{2} (1999), 299--320.

\bibitem{N26}
M.~H.~A. Newman, \emph{On the foundations of combinatorial analysis situs},
  Proc. Royal Acad. Amsterdam \textbf{29} (1926), 610--641.

\bibitem{N31}
\bysame, \emph{A theorem in combinatorial topology}, J. London Math. Soc.
  \textbf{6} (1931), 186--192.

\bibitem{P90}
U.~Pachner, \emph{Shellings of simplicial balls and p.l. manifolds with
  boundary}, Discr. Math. \textbf{81} (1990), 37--47.

\bibitem{P91}
\bysame, \emph{$\text{P.L.}$ homoeomorphic manifolds are equivalent by
  elementary shellings}, Europ. J. Combinatorics \textbf{12} (1991), 129--145.

\bibitem{S96}
R.~P. Stanley, \emph{Combinatorics and commutative algebra}, 2nd ed., Progr. in
  Math., vol.~41, Birkh\"auser, Boston, 1996.

\end{thebibliography}
\end{document}